\documentclass{amsart}
\usepackage{amscd}
\usepackage{amssymb}
\usepackage[all, knot]{xy}
\usepackage{bbm}
\usepackage[top=1.2in, bottom=1.2in, left=1.2in, right=1.2in]{geometry}
\xyoption{all}
\xyoption{arc}
\usepackage{hyperref}
\usepackage[mathscr]{euscript}

\def\oa{\overline{a}}

\def\del{\partial}

\def\HH{\operatorname{HH}}

\def\cube#1#2#3#4#5#6#7#8{
& #5 \ar[rr] \ar[dl] \ar@{-}[d] && #6 \ar[dd] \ar[dl] \\
#1 \ar[rr] \ar[dd]  & \ar[d] & #2 \ar[dd] \\
& #7 \ar@{-}[r] \ar[dl] & \ar[r] & #8 \ar[dl] \\
#3 \ar[rr] && #4 \\
}

\def\S{\Sigma}

\def\smsh{\wedge}

\def\chr{\operatorname{char}}

\def\cP{\mathcal P}

\def\cA{\mathcal A}
\def\cB{\mathcal B}
\def\cC{\mathcal C}

\def\cD{\mathcal D}

\def\dm{\operatorname{dim}}

\def\sExt{\sExt}

\def\Tor{\operatorname{Tor}}

\def\Mat{\operatorname{Mat}}
\def\End{\operatorname{End}}

\def\Spec{\operatorname{Spec}}

\def\bu{\bullet}

\def\into{\hookrightarrow}
\def\onto{\twoheadrightarrow}

\def\o#1{{\overline{#1}}}

\def\a{\alpha}
\def\b{\beta}
\def\g{\gamma}
\def\e{\epsilon}

\def\rk{\operatorname{rank}}

\DeclareMathOperator*{\colim}{colim}

\newcommand{\Q}{\mathbb{Q}}

\newcommand{\G}{\mathbb{G}}

\newcommand{\C}{\mathbb{C}}
\newcommand{\Z}{\mathbb{Z}}

\newcommand{\fp}{{\mathfrak p}}
\newcommand{\fm}{{\mathfrak m}}

\numberwithin{equation}{section}

\theoremstyle{plain} 
\newtheorem{thm}[equation]{Theorem}
\newtheorem{thm-conj}[equation]{Theorem-Conjecture}
\newtheorem{defn-conj}[equation]{Definition-Conjecture}

\newtheorem*{introthm*}{Theorem}

\newtheorem{cor}[equation]{Corollary}
\newtheorem{lem}[equation]{Lemma}
\newtheorem{prop}[equation]{Proposition}

\theoremstyle{definition}
\newtheorem{defn}[equation]{Definition}

\newtheorem{ex}[equation]{Example}

\theoremstyle{remark}
\newtheorem{rem}[equation]{Remark}

\def\Perf{\operatorname{Perf}}
\def\qPerf{\operatorname{qPerf}}

\newcommand{\Hom}{\operatorname{Hom}}

\newcommand{\xra}[1]{\xrightarrow{#1}}
\newcommand{\xla}[1]{\xleftarrow{#1}}

\newcommand{\id}{\operatorname{id}}

\newcommand{\odd}{\operatorname{odd}}

\newcommand{\ov}[1]{\overline{#1}}

\def\l{\lambda}

\def\ev{{\mathrm{even}}}
\def\odd{{\mathrm{odd}}}

\def\Hoch{\operatorname{Hoch}}

\def\oHoch{\overline{\operatorname{Hoch}}}

\def\tHoch{\operatorname{{Hoch}}}

\def\s{\sigma}
\def\d{\delta}

\def\tr{\operatorname{tr}}

\def\n{\nabla}
\def\tn{\widetilde{\nabla}}
\def\and{ \text{ and } }

\def\can{\mathrm{can}}

\def\HN{\operatorname{HN}}
\def\HP{\operatorname{HP}}

\def\op{{\mathrm{op}}}

\def\G{\Gamma}

\def\nat{\natural}
\def\inc{\mathrm{inc}}

\def\oMC{\overline{\on{MC}}}
\def\tMC{{\on{MC}}}

\def\oHN{\overline{\on{HN}}}

\def\tHN{\on{HN}}

\def\oHP{\overline{\on{HP}}}
\def\tHP{\on{HP}}

\def\O{\Omega}
\def\o{\omega}
\def\vj{{\vec{j}}}

\def\vl{{\vec{l}}}
\def\vm{{\vec{m}}}

\def\on{\operatorname}
\def\MF{\on{MF}}
\def\Ob{\on{Ob}}
\def\oG{\overline{G}}

\def\ch{\on{ch}}

\begin{document}
\begin{abstract}
Let $k$ be a field of characteristic $0$ and $\cA$ a curved $k$-algebra. We obtain a Chern-Weil-type formula for the Chern character of a perfect
$\cA$-module taking values in $\HN_0^{II}(\cA)$, the \emph{negative cyclic homology of the second kind} associated to $\cA$, when $\cA$ satisfies a
certain smoothness condition.   

\end{abstract}
\title [A Chern-Weil formula for the Chern character of a perfect curved module]{A Chern-Weil formula for the Chern character of a perfect curved module}

\author{Michael K. Brown}
\address{Department of Mathematics, University of Wisconsin, Madison, WI 53706-1388, USA}
\email{mkbrown5@wisc.edu}

\author{Mark E. Walker}
\address{Department of Mathematics, University of Nebraska, Lincoln, NE 68588-0130, USA}
\email{mark.walker@unl.edu}

\thanks{MB and MW gratefully acknowledge support from the National Science Foundation (NSF award DMS-1502553) and the Simons Foundation (grant  \#318705), respectively.} 
\maketitle

\tableofcontents

\section{Introduction}
Our goal is to obtain a Chern-Weil-type formula for the Chern character map associated to certain curved algebras.
Let us explain what we mean.

In its most classical form, the Chern character map is a ring homomorphism
from the Grothendieck group of complex vector bundles on a finite CW complex $X$ to the even part of its rational singular cohomology:
$$
ch: KU^0(X) \to \bigoplus_i H^{2i}(X, \Q).
$$
When $X$ is a smooth manifold, by extending coefficients from 
$\Q$ to $\C$ and composing with the isomorphism to the de Rham cohomology of $X$ with complex coefficients,
one obtains the {\em de Rham Chern character map} 
$$
ch_{dR}: KU^0(X) \to \bigoplus_i H^{2i}_{dR}(X;\C).
$$
For a smooth complex vector bundle $V$, Chern-Weil theory provides an explicit
formula for $ch_{dR}(V)$ in the following way. Let $\Omega^{\bullet}(X)$ denote the complexified de Rham complex of $X$, and choose a $\C$-linear connection 
$$
\nabla: \G(V) \to \G(V) \otimes_{\Omega^0(X)}
\Omega^1(X)
$$ 
on $V$. Let 
$$
R := \nabla^2: \G(V) \to \G(V) \otimes_{\Omega^0(X)} \Omega^2(X)
$$
denote the associated curvature.
Regarding $R$ as an element of $\End_{\Omega^0(X)}(\G V) \otimes_{\Omega^0(X)} \Omega^2(X)$, one has 
$$
ch_{dR}(V) = \tr(e^R) = \sum_i \frac{\tr(R^i)}{i!} \in \bigoplus_i H_{dR}^{2i}(X; \C),
$$
where
$$\tr = \tr_{\End_{\Omega^0(X)}(\G V)} \otimes \id: \End_{\Omega^0(X)}(\G V) \otimes_{\Omega^0(X)} \Omega^{\bullet}(X) \to \Omega^\bullet(X)$$ 
is the trace map.

The de Rham Chern character map and Chern-Weil formula have purely algebraic analogues. For the rest of this introduction, let $k$ be a field of
characteristic $0$ and
$X$ a smooth variety over $k$.
Grothendieck defines a very general Chern character map
$$
ch: K_0(X) \to CH^*(X) \otimes \Q
$$
taking values in the rationalized Chow groups of $X$ (in fact, this map induces an isomorphism
$K_0(X) \otimes \Q\xra{\cong} CH^*(X) \otimes \Q$).
For each $i$, there is a cycle class map
$$
CH^i(X) \otimes \Q \to H_{dR}^{2i}(X/k)
$$
taking values in the algebraic de Rham cohomology of $X$ relative to the base field $k$, and
the {\em algebraic de Rham Chern character map}  is the composition:
$$
K_0(X) \xra{ch_{dR}} \bigoplus_i H_{dR}^{2i}(X/k).
$$
As in the topological setting, $ch_{dR}(V)$ may be given an explicit formula by choosing an algebraic connection
$\nabla: V \to V \otimes \Omega^1_{X/k}$. (When $X$ is affine, every vector bundle admits
such a connection, but for vector bundles over non-affine schemes, connections typically do not exist. When $X$ is
quasi-projective, this difficulty may be circumvented by using Jouanolou's device to 
reduce to the affine case.) 

The algebraic de Rham Chern character for smooth $k$-varieties can be recast in another way using the Hochschild-Kostant-Rosenberg isomorphism
$$
\HP_0(X) \cong \bigoplus_i H_{dR}^{2i}(X/k),
$$
where $\HP$ denotes {\em periodic cyclic homology}. This gives a map
$$
ch_{HP}: K_0(X) \to \HP_0(X).
$$
In fact, $ch_{HP}$ factors through the canonical map $\HN_0(X) \to \HP_0(X)$, where $\HN_0$ denotes {\em negative cyclic
  homology}, giving a map
$$
ch_{HN}: K_0(X) \xra{ch_{HN}} \HN_0(X). 
$$

The map $ch_{HN}$ may be vastly generalized: one may replace $X$ by any $k$-linear \emph{differential graded (dg) category}, i.e., a category
enriched over complexes of $k$-vector spaces. See Section \ref{basic} below for the precise definition of a dg
category, 
and see Section~\ref{cherndgcat} for the definition of the map $ch_{HN}$ in the setting of dg categories.

There is a variant of
negative cyclic homology called {\em negative cyclic homology of the second kind} (see Section~\ref{invariants} below), and there is a canonical map
$$
\can: \HN_0(\cA) \to \HN_0^{II}(\cA)
$$
for any dg category $\cA$. By composing $ch_{HN}$ with the map $\can$, we obtain the {\em Chern character map of the second kind}
$$
ch_{HN}^{II}: K_0(\cA) \to \HN_0^{II}(\cA)
$$
for any dg-category $\cA$. The Chern character map of the second kind admits an even further generalization: we may take $\cA$ to be a {\em curved} dg
category. Hochschild-type invariants of curved dg categories (or \emph{cdg categories}, for short) have been intensively studied in recent work of, for
instance, C{\u{a}}ld{\u{a}}raru-Tu \cite{caldararu-tu}, Efimov \cite{efimov}, Platt \cite{platt}, Polishchuk-Positselski \cite{PP} and Segal \cite{segal}; we refer the
reader to Section~\ref{basic} below for background on cdg categories.

We will mostly be interested in a certain family of cdg categories called essentially smooth curved $k$-algebras, which are 
cdg categories having only one object satisfying an appropriate smoothness condition. In detail, 
an {\em essentially smooth curved $k$-algebra} is a pair $(A, h)$, where 
\begin{itemize}
\item $A$ is a commutative $\Z$- or $\Z/2$-graded $k$-algebra that is concentrated in even degrees, 
\item $h$ is an element of degree $2$, and
\item upon forgetting the grading, $A$ is an essentially smooth $k$-algebra in the classical sense.
\end{itemize}

Fix an essentially smooth curved $k$-algebra $\cA = (A, h)$. It turns out that the ordinary negative cyclic homology $\HN_*(\cA)$ is 0 when $h \ne 0$ 
\cite[Section 2.4]{PP},  but $\HN_*^{II}(\cA)$ is typically of interest: there is a variation of the Hochschild-Kostant-Rosenberg isomorphism
taking the form
$$
\e: \HN_0^{II}(\cA) \xra{\cong} H_0(\Omega^\bullet_{A/k}[[u]], ud + dh),
$$
where $(\Omega^\bullet_{A/k},d)$ is the classical de Rham complex,
$u$ is a formal parameter of degree $2$, and $dh$ is left-multiplication by the element $dh \in \Omega^1_{A/k}$. (With our indexing convention, the de Rham
differential $d$ has degree $-1$, and hence $ud$ has degree $1$.) See Section \ref{HKRsection} for more details.

Our main result gives a Chern-Weil-type formula for the Chern character map of the second kind associated to $\cA$. 
Before stating it, we recall some more terminology. A \emph{perfect right $\cA$-module} is a pair $(P, \d_P)$, where 
\begin{itemize}
\item $P$ is a finitely generated projective right graded $A$-module, and 
\item $\d_P$ is a degree $1$ $A$-linear endomorphism of $P$ such that $\d_P^2 = - \rho_{h}$, where $\rho_h$ denotes right multiplication by $h$. 
\end{itemize}
A perfect right $\cA$-module may equivalently be viewed as a ``graded matrix factorization"; see Remark \ref{mfremark} for details.

A {\em connection} on a perfect right $\cA$-module $(P, \d_P)$ is a $k$-linear map 
$$
\n: P \to P  \otimes_A  \Omega^1_{A/k}
$$
of degree $-1$ such that $\n(pa) =   \n(p)a + (-1)^{|p|}p \otimes d(a)$ for all $a \in A$ and $p \in P$. That is, a connection on $(P, \d_P)$ is just a connection on the projective $A$-module $P$ which preserves the grading; when $\G = \Z/2\Z$, such a connection is called a \emph{superconnection} \cite{quillen}. A connection on $(P, \d_P)$ extends in a natural way to a $k$-linear
endomorphism of  $P \otimes_A \Omega^\bu_{A,k}$, and, in particular, we have an $A$-linear map 
$$
\n^2 = \n \circ \n: P \to P  \otimes_A  \Omega^2_{A/k}.
$$
We define the {\em curvature} associated to such a connection to be
the $A$-linear map
$$
R: P  \to P \otimes_A \Omega^\bullet_{A/k}[[u]] 
$$
given by 
$$
R = u \n^2  + [\n, \d_P],
$$
where  $[\n, \d_P] := \n \circ \d_P + (\d_P \otimes \id) \circ \n$. 
This definition is inspired by Quillen's formula in \cite{quillen} for the Chern character of a relative topological $K$-theory class.

Using the canonical isomorphism 
$$
\Hom_A(P, P \otimes_A \Omega^\bullet_{A/k}[[u]])
\cong \End_A(P) \otimes_A \Omega^\bullet_{A/k}[[u]],
$$
we regard $R$ as an element of the right-hand side, which has a natural ring structure. In particular, powers $R^i$ make sense for $i \ge 0$.

The main result of this paper is the following:

\begin{introthm*}[see Theorem~\ref{mainthm} below]  
\label{introthm}
Let $\cA = (A, h)$ be an essentially smooth $\Z$- or $\Z/2$-graded curved algebra over a field $k$ of characteristic $0$,
and let $\cP = (P, \d_P)$ be a perfect right $\cA$-module. For any choice of 
connection $\n$ on $\cP$ with associated curvature $R$, 
the Chern-Weil formula
\begin{equation} \label{E319}
\e \circ ch^{II}_{HN}(\cP)  = \tr(\exp(-R))
\end{equation}
holds, where $\e$ is the HKR isomorphism,
$$
\tr(\exp(-R))  = \id - \tr(R) + \frac{\tr(R^2)}{2!} -  \frac{\tr(R^3)}{3!} + \cdots  \in H_0(\Omega^\bullet_{A/k}[[u]], ud + dh),
$$
and $\tr$ denotes the trace map $\tr_{\End_A(P)} \otimes \id: \End_A(P) \otimes_A \Omega^\bullet_{A/k}[[u]] \to \Omega^\bullet_{A/k}[[u]]$.
\end{introthm*}

\begin{rem}
\label{finitesum}
The alternating sum 
$$
\id - \tr(R) + \frac{\tr(R^2)}{2!} -  \frac{\tr(R^3)}{3!} + \cdots 
$$
is finite, since $\Omega^i_{A/k} = 0$ for $i > \dim(A)$. The image of $\e \circ ch^{II}_{HN}(\cP)$ is therefore contained in the image of the canonical map
$$
H_0(\Omega^\bullet_{A/k}[u], ud + dh) \to H_0(\Omega^\bullet_{A/k}[[u]], ud + dh).
$$
\end{rem}

\begin{rem}
The formula in the main theorem closely resembles Quillen's formula in \cite{quillen} for the Chern character of a relative $K$-theory class. We have not yet been able to make the relationship between the two precise.
\end{rem}

There is a commutative square
\begin{equation}
\label{HHandHN}
\xymatrix{ 
 \HN_0^{II}(\cA) \ar[d]^-{p} \ar[r]^-{\epsilon} & H_0(\Omega^\bullet_{A/k}[[u]], ud + dh) \ar[d]^-{u = 0} \\
\HH_0^{II}(\cA)  \ar[r]^-{\epsilon}_{} & H_0(\Omega^\bullet_{A/k}, dh) ,\\
}
\end{equation}
where $\HH_0^{II}$ denotes \emph{Hochschild homology of the second kind} (see Section \ref{MCunital}), $p$ is a canonical map, and the bottom map is the HKR isomorphism for Hochschild homology of the second kind, which we also denote by $\e$. Set $ch^{II}_{HH} := p \circ ch^{II}_{HN}$. Using the main theorem and
\eqref{HHandHN}, we obtain the following formula for $\epsilon \circ ch^{II}_{HH}$:

$$
\e \circ ch^{II}_{HH}(\cP) = \tr(\exp(-[\n, \d_P])) \in H_0(\Omega^\bullet_{A/k}, dh).
$$
Polishchuk-Vaintrob and Segal also obtain formulas for $\epsilon \circ ch^{II}_{HH}$ in \cite[Corollary 3.2.4]{PV} and \cite[3.2]{segal}, 
using different methods, and Platt generalizes these results to the global setting in \cite{platt}. See Example \ref{mfchern} and Remark \ref{segalremark} for comparisons between our results and those of Polishchuk-Vaintrob and Segal. 

The main theorem recovers the classical Chern-Weil formula for the Chern character of a projective module over a commutative ring:

\begin{ex}
\label{classical}
Suppose $\cA=(A, 0)$, where $A$ is a smooth $\Z$-graded $k$-algebra concentrated in degree $0$. Then one has isomorphisms
$$
\HN_0(\cA) \xra{\cong} \HN_0^{II}(\cA) \xra{\cong} H_0(\Omega^{\bullet}_{A/k}[u], ud)
= \bigoplus_{j}  H^{2j}_{dR}(A) u^j.
$$
If $P$ is a finitely generated projective $A$-module,
we may regard $P$ as an object $\cP = (P, 0) \in \Perf(\cA) = \Perf(\cA^\op)$. Choose a connection $\n$ on $P$. The main theorem implies
$$
(\epsilon \circ ch_{HN}^{II})(\cP) = \tr(\exp(-\n^2 u)) = \sum_j (-1)^j \frac{\tr(\n^{2j})}{j!}  u^j \in \bigoplus_{j}  H^{2j}_{dR}(A) u^j.
$$
This coincides, up to a sign, with the classical formula for the Chern character of a projective module over a commutative ring (see, for instance, \cite[Section 8.1]{Loday}).

Concerning the sign difference: let $\tau$ be the endomorphism of $H_0(\Omega^\bullet_{A/k}[[u]], ud + dh)$ that
sends $u$ to $-u$. We have
$$
\tau(\e(ch^{II}_{HN}(\cP)))  = \tr(\exp(R')),
$$
where $R' := u \n^2  - [\n, \d]$. So, in the setting of Example \ref{classical}, applying $\tau$ recovers the classical formula for the Chern character of a projective module on the nose. The reason we do not just define the HKR map to be $\tau \circ \epsilon$, thereby eliminating this sign discrepancy, is that applying $\tau$ destroys the $k[[u]]$-linearity of $\epsilon$.
\end{ex}

In Sections \ref{cdgc}  and \ref{Hochschild}, we recall some basic definitions and results concerning curved dg categories and their associated Hochschild-type invariants. All of the results
in these sections have appeared, or have been hinted at, in the literature. We define the Chern character of the second kind in Section \ref{chernsection}, and we prove the main theorem in Section \ref{themaintheorem}.
Section \ref{examplesection} contains several examples. In Appendix \ref{appA}, we justify our choice of sign convention in the Hochschild complex of the second kind; our convention differs from the one in \cite{PP}. Appendix \ref{appendixtechnical} is devoted to a proof of a key technical result (Theorem~\ref{existence}) needed for
the proof of the main theorem.

\vskip\baselineskip
\noindent{\bf Notational conventions:} 
\begin{itemize}

\item $k$ denotes a field.

\item Throughout, we wish to work simultaneously with both $\Z$-graded and $\Z/2$-graded $k$-vector spaces, algebras, etc.,
and so we let $\G$ denote either $\Z$ or $\Z/2$ and refer to $\G$-graded $k$-vector spaces, $k$-algebras, etc. 
We adhere to cohomological indexing, so that the degree $j$ part of a $\Z$-graded vector space $V$ is written $V^j$. When subscripts are used, it is understood
that $V_j$ denotes $V^{-j}$.

\item A {\em differential} on a $\G$-graded $k$-vector space is a homogeneous endomorphism $d$ of degree $1$. Note that we do {\em not}
  assume $d^2 = 0$ in general.

\item An element of a $\G$-graded $k$-vector space $V$ shall always be taken to mean a homogeneous element, unless otherwise stated. 

\item $[ - , - ]$ means graded commutator: for any $\G$-graded ring $A$ and (homogeneous) elements $x, y \in A$, we set
$[x,y] := xy - (-1)^{|x||y|} yx$.  We say a graded ring is {\em commutative} if 
$[x,y] = 0$ for all (homogeneous) elements $x$ and $y$, and $x^2 = 0$ for all elements $x$ of odd degree (the latter is redundant if $2$ is invertible in $A$).

\item If $V$ is a $\G$-graded $k$-vector space, the {\em suspension} $\S V$  of $V$ is given by setting $(\S V)^i = V^{i + 1}$. If $V$ is equipped with a
  differential $d$, then $\S V$ is equipped with the differential $-d$. We write $s: V \to \S V$ for the evident degree $-1$ map sending an element
$v \in V^i$ to itself but regarded as an element of 
  $(\S V)^{i-1}$.  

\end{itemize}
\vskip\baselineskip
\noindent{\bf Acknowledgements.} We thank the referee for catching an error in the original proof of Theorem \ref{existence}.

\section{Curved differential graded categories}
\label{cdgc}
We recall some general notions  concerning curved differential graded categories, closely following Polishchuk-Positselski \cite[Section 1]{PP}.

\subsection{Basic definitions}
\label{basic}

\begin{defn}
\label{cdgdefn} Let $k$ be a field, and let $\G$ be either  $\Z$ or $\Z/2$. 
A \emph{$k$-linear curved differential $\G$-graded category} $\cA$ (or just ``cdg category'', for short) consists of the following data:
\begin{itemize}

\item a collection of objects $\Ob(\cA)$;

\item for each pair  $X,Y \in \Ob(\cA)$ of objects, a $\G$-graded $k$-module $\Hom_{\cA}(X,Y)$ equipped with a $k$-linear differential $\d_{X,Y}$ (i.e., a
  degree one endomorphism that need not square to $0$);

\item a composition law: that is, for any $X,Y,Z \in \Ob(\cA)$, a morphism of $\G$-graded $k$-modules
$$
\mu_{X,Y,Z}: \Hom_{\cA}(Y,Z) \otimes_k \Hom_{\cA}(X,Y) \to \Hom_{\cA}(X,Z);
$$
\item for all $X \in \Ob(\cA)$, an identity morphism $1_X \in \Hom_{\cA}(X,X)$;
\item for every $X \in \Ob(\cA)$, a degree 2 element $h_X \in \Hom_{\cA}(X,X)$ called a \emph{curvature element}.
\end{itemize}
We will write $\d$ for $\d_{X,Y}$ and $\mu$ for $\mu_{X,Y,Z}$ when the underlying objects are understood. We will sometimes also use the usual $ - \circ - $ notation for composition, rather than $\mu$. 

The above data is required to satisfy the following conditions:
\begin{itemize}
\item For all $W, X, Y, Z \in \Ob(\cA)$, the diagram
$$
\xymatrix{
		\Hom_{\cA}(Y, Z) \otimes \Hom_{\cA}(X, Y) \otimes \Hom_{\cA}(W, X) \ar[r]^-{\id \otimes \mu}\ar[d]^-{ \mu \otimes \id}&  \Hom_{\cA}(Y, Z) \otimes \Hom_{\cA}(W, Y) \ar[d]^{\mu}\\
	\Hom_{\cA}(X, Z) \otimes \Hom_{\cA}(W, X) \ar[r]^-{\mu} &\Hom_{\cA}(W, Z)}
$$
commutes.

\item For each $X \in \Ob(\cA)$, the compositions
$$
\Hom_{\cA}(X,Y) \xrightarrow{\cong}  \Hom_{\cA}(X,Y) \otimes k \xrightarrow{ \id \otimes 1_X} \Hom_{\cA}(X,Y) \otimes \Hom_{\cA}(X,X) \xrightarrow{\mu}
\Hom_{\cA}(X,Y)
$$
and
$$
\Hom_{\cA}(X,Y) \xrightarrow{\cong} k  \otimes  \Hom_{\cA}(X,Y) \xrightarrow{ 1_Y \otimes \id} \Hom_{\cA}(Y,Y) \otimes \Hom_{\cA}(X,Y) \xrightarrow{\mu}
\Hom_{\cA}(X,Y)
$$
are equal to the identity.

\item For all objects $X$, $Y$, and $Z$, the square
$$
\xymatrix{
\Hom_{\cA}(X, Y) \otimes \Hom_{\cA}(Y, X) \ar[rr]^{\delta \otimes \id + \id \otimes \delta}
\ar[d]^{\mu}  
&&  \Hom_{\cA}(X, Y) \otimes \Hom_{\cA}(Y, X) \ar[d]^{\mu} \\
\Hom_{\cA}(X,Z) \ar[rr]^{\delta } && \Hom_{\cA}(X,Z). 
}
$$
commutes. Explicitly, given elements
$f \in \Hom_{\cA}(X, Y)$ and $g \in \Hom_{\cA}(Y, Z)$, we have
$$
\d(\mu(g \otimes f)) = \mu(\d(g) \otimes f) + (-1)^{|g|}\mu(g \otimes \d(f)).
$$

\item For all objects $X$ and $Y$
and all $f \in \Hom_{\cA}(X, Y)$, we have
$$
\d^2(f) = \mu(h_Y \otimes f) - \mu(f \otimes h_X).
$$

\item $\d(h_X) = 0$ for all $X \in \Ob(\cA)$.
\end{itemize}
A cdg category in which  $h_X = 0$ for all $X$ is called a \emph{$k$-linear differential $\Gamma$-graded category}, or just a {\em dg category}
for short.
\end{defn}

\begin{ex}
A basic example of a cdg category is given by \emph{precomplexes of $k$-vector spaces}:
\begin{itemize}
\item an object of $\on{Pre}(k)$ is a $\G$-graded $k$-vector space $V$ equipped with a $k$-linear endomorphism  $d_V$ of degree $1$;

\item $\Hom_{\on{Pre}(k)}(W, V)$ is the $\G$-graded $k$-vector space whose degree $n$ component is $\prod_j \Hom_k(W^j, V^{j+n})$, 
and it is equipped with the differential which sends $f$ to $d_{V} f -
(-1)^{|f|}f d_W$;

\item the curvature element $h_V \in \Hom_{\on{Pre}(k)}(V, V)$ is
  $d_V^2$; and

\item the unit and composition maps are the obvious ones.

\end{itemize}

\end{ex}

\begin{defn}
Let $\cA$ and $\cA'$ be cdg categories. A \emph{cdg functor} $F: \cA \to \cA'$ is given by the following data:
\begin{itemize}
\item a function $F: \Ob(\cA) \to \Ob(\cA')$;
\item for all $X,Y \in \Ob(\cA)$, a morphism of graded $k$-vector spaces 
$$
F_{X,Y}: \Hom_{\cA}(X,Y) \to \Hom_{\cA'}(F(X),F(Y));
$$
\item for all $X \in \Ob(\cA)$, a degree 1 element $a_X \in \Hom_{\cA'}(F(X), F(X))$.
\end{itemize}

We will write $F$ for $F_{X,Y}$ when the underlying objects are understood. These data must satisfy the following conditions:
\begin{itemize}
\item $F$ preserves composition: that is, for all $X, Y, Z \in \Ob(\cA)$, 
$$
\mu_{F(X), F(Y), F(Z)} \circ (F_{Y,Z} \otimes F_{X,Y}) = F_{X,Z} \circ \mu_{X,Y,Z};
$$

\item $F$ preserves identities: i.e., for each $X \in \Ob(\cA)$ we have
$$
F_{X,X} \circ 1_X = 1_{F(X)};
$$

\item for all $X, Y \in \Ob(\cA)$ 
and  $f \in \Hom_{\cA}(X, Y)$, 
$$
F(\d(f)) =\d(F(f)) + \mu(a_Y \otimes F(f)) - (-1)^{|f|} \mu(F(f) \otimes a_X);
$$

\item for all $X \in \Ob(\cA)$, 
$$
F(h_X) = h_{F(X)} + \d (a_X) + a_X^2.
$$

\end{itemize}

A cdg functor is called \emph{strict} if $a_X=0$ for all $X$. A strict cdg functor between dg categories is called a \emph{dg functor}. 
 
A (strict) \emph{quasi-cdg functor}, or (strict) \emph{qdg functor}, $F: \cA \to \cA'$ consists of the same data as a (strict) cdg functor, and it satisfies the
same conditions as a cdg functor except for the last one:  we do not require the equation relating $F(h_X)$ and $h_{F(X)}$ to hold. Composition of qdg functors $F: \cA \to \cA'$ and $G: \cA' \to \cA''$ is defined in the obvious way on objects and morphisms, and, for each $X \in \cA$, the distinguished degree $1$ element of $\Hom_{\cA''}((G \circ F)(X), (G \circ F)(X))$ is given by $G(a_X) + b_{F(X)}$, where $\{a_X\}_{X \in \cA}$ and $\{b_{X'} \}_{X' \in \cA'}$ are the families of distinguished elements associated to $F$ and $G$.
\end{defn}

\begin{ex} 
\label{functorexample}
A cdg category with just one object is a called a  \emph{curved differential graded algebra}, or {\em cdga}. This amounts to a triple $(A,d,h)$, where $A$ is a
unital $\G$-graded  $k$-algebra, 
$d$ is a degree one $k$-linear endomorphism of
$A$, and $h$ is a degree 2 element of $A$ such that 
\begin{itemize}
\item $d$ satisfies the Leibniz rule: $d(ab) = d(a) b + (-1)^{|a|} a d(b)$;
\item $d^2(a) = ha - ah$;
\item $d(h) = 0$. 
\end{itemize}
A {\em curved algebra} is a  cdga with trivial differential, i.e. a $\G$-graded $k$-algebra $A$ equipped with a specified degree $2$ element
$h$ such that $ha = ah$ for all $a \in A$. We will often write a cdga  as a triple $(A, d,h)$ and a curved algebra as a pair $(A,h)$.

Unravelling the definitions, a morphism $(A,d,h) \to (A', d',h')$ of cdga's is given by a pair $(\rho, \b)$, with $\rho: A \to A'$ a homomorphism of unital 
$\G$-graded $k$-algebras and $\b \in A'$ a degree one element, such that 
\begin{itemize}
\item $\rho(d(a)) - d'(\rho(a)) = [\b, \rho(a)]$ for all $a \in A$, and
\item $\rho(h)  = h' + d'(\b) + \b^2$.
\end{itemize}
A strict morphism $(A,d,h) \to (A', d',h')$ is given by a map $\rho: A \to A'$ of $\G$-graded $k$-algebras such that $\rho \circ d = d' \circ \rho$ and
$\rho(h) = h'$.

As a simple example of a non-strict morphism of cdga's, take $A$ to be an exterior algebra on a set of degree $1$ variables, and consider it as a curved algebra with trivial curvature. Let $\beta \in A$ be a degree $1$ element. Then
$$(\id_A, \beta): (A, 0) \to (A, 0)$$
is a non-strict morphism.
\end{ex}

\begin{ex}
The following example of a cdga explains the presence of the word ``curved" in the terminology introduced above. Let $X$ be a smooth manifold and $V$ a smooth complex vector bundle
on $X$.
Let $\Omega^{\bullet}(X)$ denote the complexified de Rham complex of $X$, and let $\n$ be a connection on $V$, i.e. a $\C$-linear map
$$
\Gamma(V) \to  \G(V) \otimes_{\Omega^0(X)} \Omega^1(X)
$$
such that $\n(fs) = f\n(s)  + s \otimes df$ for $s \in \Gamma(V)$ and $f \in \Omega^0(X)$. $\n$ extends uniquely to a $\C$-linear derivation $\widetilde{\n}$ of $\G(V) \otimes_{\Omega^0(X)} \Omega^{\bullet}(X)$. Set $R:=\widetilde{\n}^2|_{\G(V)}$; $R$ is called the \emph{curvature} of $\n$. Note that $R$ may be identified with an element of the $\C$-algebra $\End_{\Omega^0(X)}(\G(V)) \otimes_{\Omega^0(X)} \Omega^{\bullet}(X)$.

The connection $\n$ on $V$ induces a connection $\n'$ on $\End(V)$ given by $[\n, -]$ which yields a derivation $\widetilde{\n'}$ of $\End_{\Omega^0(X)}(\G(V))  \otimes_{\Omega^0(X)} \Omega^{\bullet}(X)$. The triple
$$
(\End_{\Omega^0(X)}(\G(V)) \otimes_{\Omega^0(X)} \Omega^{\bullet}(X), \widetilde{\n'}, R)
$$
is a curved differential $\C$-algebra \cite[Section 4.1]{pos}. Notice that the curvature of this cdga coincides with the curvature of $\n$.
\end{ex}
\begin{defn}
If $\cA$ is a cdg category, its \emph{opposite}, $\cA^\op$, is defined as follows:
\begin{itemize}
\item $\Ob(\cA^\op) = \Ob(\cA)$;
\item $\Hom_{\cA^\op}(X, Y) := \Hom_{\cA}(Y, X)$, equipped with the differential $\d_{Y, X}$;
\item for $X, Y, Z \in \Ob(\cA^\op)$ and elements $f \in \Hom_{\cA^\op}(X,Y)$ and  $g \in \Hom_{\cA^\op}(Y,Z)$, $\mu^\op_{X,Y,Z}(g \otimes f) = (-1)^{|f||g|}\mu_{Z,Y,X}(f \otimes g)$;
\item for all $X \in \Ob(\cA^\op)$, the curvature element in $\Hom_{\cA^\op}(X, X)$ is $-h_X$.
\end{itemize}
Given a cdg functor $F: \cA \to \cA'$, the \emph{opposite} functor $F^\op:\cA^\op \to (\cA')^\op$ is defined in the evident way on objects and morphisms, and,
for each object $X \in \cA$, the distinguished degree $1$ element of $\Hom_{(\cA')^\op}(F(X), F(X))$ is $-a_X$.  
\end{defn}

\begin{ex} \label{ExOp}
If $\cA = (A, d, h)$ is a cdga, $\cA^\op$ is the cdga $(A^\op, d, -h)$, where $A^\op$ has the multiplication rule $\star$ given by $x \star y = (-1)^{|x||y|}
yx$. Given a morphism $(\rho, \b): (A, d, h) \to (A', d', h')$ of cdga's,
$$
(\rho, -\b): (A^\op, d, -h) \to ((A')^\op, d', -h')
$$
is the induced opposite map.
\end{ex}

Polishchuk-Positselski introduce in \cite{PP} the notion of a \emph{pseudo-equivalence} of cdg categories, and it will be a useful notion in this paper:

\begin{defn} \label{pseudo}
Let $\cA$ be a cdg category, and fix an object $X \in \cA$. 
\begin{itemize}
\item $X$ is called the \emph{direct sum} of a family of objects $\{X_\alpha\}_{\a \in I}$  of $\cA$ if there exist degree $0$ morphisms $i_\alpha : X_\alpha \to X$ for each $\a
  \in I$  such that 
\begin{enumerate}
\item $\d_{X_\a, X}(i_\alpha) = 0$ for all $\alpha$;
\item the induced map $\Hom_\cA(X, Y) \to \prod_{\alpha}\Hom_{\cA}(X_\alpha, Y)$ is an isomorphism of $\G$-graded $k$-vector spaces for all $Y$.
\end{enumerate}
\end{itemize}
In this situation, each $X_\alpha$ is referred to as a {\em direct summand} of $X$.

\begin{itemize}
\item Fix a degree $1$ endomorphism $\tau$ of $X$. An object $Y \in \cA$ is called a \emph{twist of $X$ with $\tau$} if there exist degree $0$ morphisms $i: X \to
  Y$ and $j: Y \to X$ such that $j\circ i = \id_X$, $i \circ j = \id_Y$, and $j \circ \d_{X,Y}(i) = \tau$. 
 
\item Given $n \in \Z$, an object $Y \in \cA$ is called an \emph{$n^{\on{th}}$ shift} of $X$ if there exist morphisms $i: X \to Y$ and $j : Y \to X$ of degrees $n$ and
  $-n$, respectively, such that $ji = \id_X$, $ij = \id_Y$, and $\d_{X,Y}(i) = 0 = \d_{Y,X}(j)$. 

\item A cdg functor $F: \cA \to \cA'$ is a \emph{pseudo-equivalence} if $F$ induces isomorphisms 
$$
\Hom_{\cA}(Y, Z) \xra{\cong} \Hom_{\cA'}(F(Y),F(Z)) 
$$
for all $Y, Z \in \cA$, and every object of $\cA'$ may be obtained from objects of the form $F(X)$ for $X \in \Ob(\cA)$  via the operations of finite direct sum, shift, twist, and
passage to a direct summand. 
\end{itemize}
\end{defn}

\subsection{Modules over cdg categories}
\label{modules}

Fix a cdg category $\cA$.
\begin{defn}
A \emph{left $\cA$-module} (resp. \emph{left quasi-$\cA$-module}) is a strict cdg (resp. qdg) functor
$$
F: \cA \to \on{Pre}(k).
$$
Right $\cA$-modules and quasi-$\cA$-modules are defined similarly, replacing $F: \cA \to \on{Pre}(k) $ with $F: \cA^\op \to \on{Pre}(k)$.
\end{defn}

\begin{ex}
\label{representable}
Fix an object $X$ of $\cA$. The 
representable functor given, on objects, by $Y \mapsto \Hom_{\cA}(X, Y)$ (resp. $Y \mapsto \Hom_{\cA}(Y, X)$) is a left (resp. right) quasi-module. Each is a module provided
$h_X = 0$.
\end{ex}

\begin{defn}
\label{perfectquasi}
A (left or right) quasi-$\cA$-module is called \emph{perfect} if the induced ordinary functor between the $k$-linear $\G$-graded categories underlying $\cA$ and $\on{Pre}(k)$
is, up to a shift, a direct summand of a finite direct sum of representable functors.
\end{defn}

Let us unravel these definitions in the case where $\cA$ is a cdga:

\begin{ex}  \label{ex630}
Suppose $\cA$ is a cdga $(A, d, h)$. Then we identify a left quasi-module $F: \cA \to \on{Pre}(k)$ with the value of $F$ on the unique object of $\cA$. That is, 
a left quasi-module over $\cA$ is pair $(M,\d_M)$, where $M$ is a graded left $A$-module, and $\d_M: M \to M$ is a $k$-linear map of degree $1$
satisfying the Leibniz rule:
$$
\d_M(a m) = d(a) m + (-1)^{|a|} a \d_M(m) \text{ for all $a \in A$, $m \in M$.}
$$
A left quasi-module $(M, \d_M)$ is a left module if
$\d_M^2 = \l_h$, where $\l_h$ denotes left multiplication by the curvature element $h$ of $A$. A left quasi-module $(M, \d_M)$ is perfect if $M$ is finitely
generated and projective as a graded left $A$-module.

Using Example \ref{ExOp}, one sees that a right quasi-module/module over $(A,d,h)$ is the same thing as a left quasi-module/module over $(A^\op, d,-h)$. In other words, a
right quasi-module over $(A,d,h)$ consists of  a graded right $A$-module $M$  and a $k$-linear map $\d_M: M \to M$ of degree one that satisfies $\d_M(ma) =
\d_M(m) a + (-1)^{|m|} m d(a)$. Such a pair $(M, \d_M)$ is a right module if $\d_M^2 = - \rho_h$, where $- \rho_h(m) = -mh$. 

Notice that there is no natural sense in which $(A, d)$ is a left or right $\cA$-module, unless $h = 0$, but $(A, d)$ is both a left and right
quasi-$\cA$-module in the standard ways.  Henceforth, we will abuse
notation slightly and say that $\cA$ is a (left and right) quasi-$\cA$-module. 
\end{ex}

\begin{rem}
\label{linear}
Recall that, given a $\G$-graded ring $R$ and $\G$-graded left $R$-modules $N, N'$, a degree $e$ map $f: N \to N'$ of $\G$-graded abelian groups is \emph{$R$-linear} if 
$$
f(rn) = (-1)^{|r|e}rf(n)
$$ 
for all elements $n \in N$ and $r \in R$. Observe that, if $\cA$ is a curved algebra and $(M, \d_M)$ is a left $\cA$-module, then $\d_M$ is $A$-linear.
\end{rem}

It is observed in \cite[Section 1]{PP} that left quasi-$\cA$-modules form a cdg category $\on{qMod}(\cA)$, and left $\cA$-modules form a dg category
$\on{Mod}(\cA)$. Let $\qPerf(\cA)$ denote the cdg subcategory of $\on{qMod}(\cA)$ consisting of perfect left quasi-modules, and let $\Perf(\cA)$ denote the dg
subcategory of $\on{Mod}(\cA)$ consisting of perfect left modules. 

\begin{ex}
\label{perfexample}
When $\cA$ is a cdga $(A, d, h)$, the cdg category $\qPerf(\cA)$ is defined as follows:

\begin{itemize}

\item Objects are pairs $(P, \d_P)$ that are isomorphic to summands of quasi-modules of the form 
$\S^{n_1}A \oplus \cdots \oplus \S^{n_r}A$ for some $r \geq 0$ and $n_i \in \Z$. Note that, in particular, $P$ is
a finitely generated graded projective left $A$-module. 

\item The morphisms $(P, \d_P) \to (P', \d_{P'})$ are given by the $\G$-graded $k$-module $\Hom_A(P, P')$ of $A$-linear maps from $P$ to $P'$,
equipped with the differential $\del = \del_{P,P'}$  defined by 
$$
\del(g) = \d_{P'} \circ g - (-1)^{|g|} g \circ \d_P
$$
on homogeneous maps $g$;
\item composition is defined in the obvious way;
\item the curvature element associated to an object $(P, \d_P)$ is $h_P = \d_P^2 - \l_h$, where $\l_h$ denotes left multiplication by the element $h$.
\end{itemize}

Notice that an object $\cP \in \qPerf(\cA)$ has trivial curvature if and only if $\d_P^2 = \l_h$, which is precisely the extra property required of a
quasi-module to make it a module. Thus,
$\Perf(\cA)$ is the full cdg subcategory of $\qPerf(\cA)$ consisting 
of those objects whose endomorphism dga's have trivial curvature. In particular, $\Perf(\cA)$ is a dg category.
\end{ex}

\begin{rem} \label{rem72c}
If $\cP = (P, \d_P)$ is a perfect {\em right} quasi-module over $\cA = (A,d,h)$ (i.e., an object in $\qPerf(\cA^\op)$), 
then its endomorphism cdga is
$(\End_A(P), [\d_P,-], \d_P^2 + \rho_h)$, where $\rho_h$ denotes right multiplication by $h$.
\end{rem}

\begin{rem} 
\label{mfremark}
If $\cA = (A, h)$ is a curved algebra, and $A$ is commutative upon forgetting the grading, a perfect right $\cA$-module $(P, \d_P)$ is precisely the data of a ``graded matrix 
  factorization of $-h$". In detail, set $P_\ev = \bigoplus_j P^{2j}$ and $P_\odd = \S^{-1} \left(\bigoplus_j P^{2j+1}\right)$, 
each of which is a graded $A$-module in the evident way. Then the
  differential $\d_P$ determines a pair of $A$-linear maps
$$
\a: P_\odd \to   P_\ev \and \b: P_\ev \to  \S^2 P_\odd 
$$
such that $\b \circ \a = -h$ and $\S^2(\a) \circ \b = -h$. In particular, when $\G = \Z/2\Z$, the category of perfect right $\cA$-modules is
identical to the differential 
$\Z/2$-graded category $\MF(A, -h)$ of matrix factorizations of $-h$, as defined in \cite[Definition 2.1]{dyc}.
\end{rem}

\begin{ex}
\label{perfpseudo}
Let $\cA$ be a cdga, let $\cP=(P, \d_P)$ be an object of $\Perf(\cA)$, and define $\cP^\nat:=(P, 0) \in \qPerf(\cA)$. Then $\cP^\nat$ and $\cP$ are twists of one another. Thus, letting $\{\cP, \cP^\nat\}$ denote the full cdg subcategory of $\qPerf(\cA)$ consisting of these two objects, we see that the inclusions
$$
\{\cP\} \into \{\cP, \cP^\nat\}
$$
$$
\{\cP^\nat \} \into \{\cP, \cP^\nat\}
$$
are pseudo-equivalences of cdg categories. This fact will be used in the proof of Proposition~\ref{comparison}.
\end{ex}

For any cdg category $\cA$, there is a ``quasi-Yoneda'' embedding 
\begin{equation} \label{E73a}
qY_\cA: \cA \into \on{qPerf}(\cA^\op)
\end{equation} 
given, on objects, by $X \mapsto \Hom_{\cA}(-, X)$. If $\cA$ has trivial curvature, then the image of $qY_\cA$ lands in the full subcategory $\Perf(\cA^\op)$,
giving the more classical Yoneda embedding. We will write $Y_\cA$ for this functor.

\begin{ex} 
\label{endAop}
Suppose $\cA$ is a cdga $(A, d, h)$, which we regard as a cdg category with one object $*$ whose endomorphism cdga is $A$. 
Under the identification of Example \ref{ex630}, the object $qY_\cA(*) \in \on{qPerf}(\cA^\op)$ is given by the pair $(A, d)$, where $A$ is regarded as a  right module over itself.
Moreover, the endomorphism cdga $\End_{\qPerf(\cA^\op)}(qY_\cA(*))$
has as its underlying algebra the collection $\End_A^{\mathrm{right}}(A)$ of right $A$-module endomorphisms of $A$ with multiplication given
by composition.
There is a strict isomorphism of cdga's 
$$
\cA  \xrightarrow{\cong} \End_{\qPerf(\cA^\op)}(qY_\cA(*)) = (\End_A^{\mathrm{right}}(A), [d, -], \rho_h)
$$
that sends $a \in A$ to $\l_a$, where $\l_a(x) = ax$.
\end{ex}

The following result is proven in \cite[Lemma A, page 5319]{PP}:

\begin{prop}[Polishchuk-Positselski]
\label{PPpseudo}
For any cdg category $\cA$, the inclusions 
$$
\Perf(\cA^\op) \hookrightarrow \qPerf(\cA^\op) \hookleftarrow \cA
$$
are pseudo-equivalences.
\end{prop}

\section{Mixed Hochschild complexes and the HKR theorem}
\label{Hochschild}
In this section, we define the mixed Hochschild complexes of the first
and second kinds associated to a curved differential $\G$-graded category, and we discuss some of their
properties. All of the results in this section are stated, or at least alluded to, in the literature; see, for instance, 
\cite{efimov}, \cite{PP}, or \cite{shkly}. 

\subsection{Mixed Hochschild complexes}
\label{MCunital}

We begin by recalling the notion of a mixed complex. Recall that $k$ denotes a field, and $\G \in \{\Z, \Z/2\}$. Let $k[B]$ be the commutative $\G$-graded $k$-algebra freely generated by  
a degree $-1$ element $B$, so that $k[B]$ is a two dimensional $k$-vector space,  
spanned by $1$ and $B$, with 
$B^2 = 0$. 
A {\em mixed complex} over $k$ is a dg-$k[B]$-module, where $k[B]$ is considered as a dg algebra with trivial differential. Equivalently, a mixed complex is a $\G$-graded $k$-vector space $M$
equipped with endomorphisms $b$ and $B$ of degrees $1$ and $-1$, respectively, such that $b^2 = 0 = B^2$ and $[b,B] = bB + Bb = 0$. We will write mixed
complexes as triples $(M, b, B)$. The notions of morphisms and quasi-isomorphisms of mixed complexes are the standard ones for
dg-$k[B]$-modules. In detail, a morphism $(M, b, B) \to (M', b', B')$ of mixed complexes is given by a graded map of $k$-vector spaces of degree $0$ that
commutes with $b,b'$ and $B, B'$. Such a morphism is a quasi-isomorphism if it is so regarded as a map $(M,b) \to (M', b')$ 
of $k$-complexes. 

We associate to any curved differential $\G$-graded category $\cA$ two mixed complexes, $\tHoch(\cA)$ and  $\tHoch^{II}(\cA)$ (cf. \cite{efimov}),
called {\em the (ordinary) Hochschild complex} and  {\em Hochschild complex of the second kind}.
We give detailed constructions when $\cA$ is a cdga $(A,d,h)$ (i.e., a cdg category with just one object), and then we sketch how to extend the definitions to
the general setting.

For a cdga $\cA = (A,d,h)$, define the $\G$-graded $k$-vector space
$$
\tHoch(\cA) = \bigoplus_{n \geq 0} A \otimes_k (\S A)^{\otimes n},
$$ 
where $\S A$ is the shift of $A$, i.e. $(\S A)^i = A^{i+1}$. 
As a standard bit of shorthand, for elements $a_0, a_1, \dots, a_n \in A$, we set
$$
a_0 [a_1| \dots| a_n] := a_0 \otimes s a_1 \otimes \cdots  \otimes s a_n  \in A \otimes_k (\S A)^{\otimes n}
$$
where $s: A \to \S A$ is the canonical degree $-1$ map.
Thus,  an element of degree $e$ in $\tHoch(\cA)$ is a finite sum of elements of the form
$a_0 [a_1| \dots| a_n]$ such that $|a_0| + |a_1| + \cdots + |a_n| - n = e$. 

Equip $\tHoch(\cA)$ with a differential  $b = b_2 + b_1 + b_0$, where
\begin{equation} \label{b2formula}
\begin{aligned}
b_2(a_0 [a_1| \dots| a_n]) &  =  
 (-1)^{|a_0|}  a_0a_1[a_2 | \cdots | a_n] \\
& + \sum_{j=1}^{n-1}  (-1)^{|a_0| + \cdots + |a_j| - j}a_0[a_1| \cdots | a_j a_{j+1}| \cdots | a_n] \\
& - (-1)^{(|a_n| - 1)(|a_0| + \cdots + |a_{n-1}| -(n-1))} a_na_0[a_1| \cdots |a_{n-1}], \\
\end{aligned}
\end{equation} 
\begin{equation} \label{b1formula}
\begin{aligned} 
b_1(a_0 [a_1| \dots| a_n]) 
&  =   d_A(a_0) [a_1| \cdots | a_n] \\ 
&  +  \sum_{j=1}^n (-1)^{|a_0| + \cdots  + |a_{j-1}|  -  j} a_0 [a_1| \cdots |d(a_j)| \cdots | a_n], \\ 
\end{aligned}
\end{equation}
and
\begin{equation} \label{b0formula}
\begin{aligned}
b_0(a_0 [a_1| \dots| a_n]) 
&  =   \sum_{j=0}^n (-1)^{|a_0| +|a_1| + \cdots + |a_{j}| - j } a_0 [a_1| \cdots |a_{j}| h|a_{j+1}| \cdots | a_n]. \\ 
\end{aligned}
\end{equation}

Define $\tHoch^{II}(\cA)$ to be the $\G$-graded $k$-vector space
$$
\tHoch^{II}(\cA) = \prod_{n \geq 0} A \otimes_k (\S A)^{\otimes n}.
$$ 
The relationship between $\tHoch(\cA)$ and $\tHoch^{II}(\cA)$ can be described in terms of 
the descending filtration $F^0 \supseteq F^1 \supseteq \cdots$ on $\tHoch(\cA$) 
defined by
\begin{equation} \label{E721}
F^j = \bigoplus_{n \geq j} A \otimes_k (\S A)^{\otimes n}.
\end{equation}
The complex $\tHoch^{II}(\cA)$ is the completion of $\tHoch(\cA)$ for the topology defined by this filtration. 
Since the differential $b$ on $\tHoch(\cA)$ is continuous with respect to this topology, it induces a differential on $\tHoch^{II}(\cA)$, which we also write as 
$b$.

\begin{rem}
\label{PPsigns}
The complex $\tHoch(\cA)$ (resp. $\tHoch^{II}(\cA)$) also appears in \cite[Section 2.4]{PP} under the name $\Hoch_{\bullet}^{\oplus}(\cA, \cA)$
(resp. $\Hoch_{\bullet}^{\sqcap}(\cA, \cA)$), but with a differential which differs from ours by a sign. The map  
$$a_0[a_1 | \cdots | a_n] \mapsto (-1)^{\sum_{j = 0}^{n-1}(n-j)|a_j|}a_0[a_1 | \cdots | a_n]$$
yields isomorphisms of complexes of $k$-vector spaces
$$
\Hoch_{\bullet}^{\oplus}(\cA, \cA) \xra{\cong} \tHoch(\cA) \and
\Hoch_{\bullet}^{\sqcap}(\cA, \cA) \xra{\cong} \tHoch^{II}(\cA).
$$
The reason for the discrepancy is that the underlying graded $k$-vector space of $\Hoch_{\bullet}^{\oplus}(\cA, \cA)$ is
$\bigoplus_{n \ge 0} A \otimes \S^n(A^{\otimes n})$, rather than $\bigoplus_{n \ge 0} A \otimes_k (\S A)^{\otimes n}$ (and similarly for $\Hoch_{\bullet}^{\sqcap}(\cA,
\cA)$). Appendix~\ref{appA} contains an explanation 
for the signs appearing in our formulas for $b_0$, $b_1$, and $b_2$.
\end{rem}

More generally, if $\cA$ is any cdg category whose objects form a set, define $\tHoch(\cA)$ and $\tHoch^{II}(\cA)$ to be the graded vector spaces
$$
\bigoplus_{X \in \cA} \Hom_{\cA}(X) \oplus \bigoplus_{n \ge 1} \bigoplus_{X_0, \dots, X_n \in \cA} \Hom_{\cA}(X_1, X_0) \otimes \S\Hom_{\cA}(X_2, X_1) \otimes
\cdots \otimes \S\Hom_{\cA}(X_0, X_n)
$$
and
$$
\bigoplus_{X \in \cA} \Hom_{\cA}(X) \oplus \prod_{n \ge 1} \bigoplus_{X_0, \dots, X_n \in \cA} \Hom_{\cA}(X_1, X_0) \otimes \S\Hom_{\cA}(X_2, X_1) \otimes
\cdots \otimes \S\Hom_{\cA}(X_0, X_n),
$$
equipped with differentials defined using the same formulas, suitably interpreted, as in the cdga case.
When $\cA$ is \emph{essentially small}, so that the isomorphism classes of objects in the $\G$-graded category
underlying $\cA$ form a set (see \cite[Section 2.6]{PP}), we define
$\tHoch(\cA)$ and $\tHoch^{II}(\cA)$ by first replacing $\cA$ with a full subcategory consisting of a single object from each isomorphism class.

\begin{defn} The {\em (ordinary) Hochschild homology groups} and {\em Hochschild homology groups of the second kind} of a cdg category $\cA$ are given by
$$
\HH_q(\cA) := H^{-q} \tHoch(\cA) \and \HH^{II}_q(\cA) := H^{-q} \tHoch^{II}(\cA).
$$
\end{defn}

\begin{rem} \label{rem629}
 As discussed in \cite[Section 2.4]{PP}, ordinary Hochschild homology has limited value in the presence of non-trivial curvature. For example if $\cA = (A, d, h)$ is a cdga
 with $h \ne 0$, then $\HH_*(\cA) = 0$. 
\end{rem}

\begin{rem} 
We have chosen to define Hochschild homology of the second kind in terms of the complex $\tHoch^{II}$, 
but, in \cite{PP}, it is defined in a different manner, using $\Tor$
of the second kind. By \cite[Proposition A in  Section 2.4]{PP}, their definition coincides with the one given here.
\end{rem}

It will be useful to also have ``reduced'' versions of Hochschild homology and Hochschild homology of the second kind. If $\cA = (A,d,h)$ is a 
cdga, let $\overline{A}$ denote the graded $k$-vector space $A/k \cdot 1_A$. Set
$$
\oHoch(\cA) :=
\bigoplus_{n \geq 0} A \otimes_k ( \S \overline{A})^{\otimes n},
$$
and similarly for $\oHoch^{II}(\cA)$. These are quotients of 
$\tHoch(\cA)$ and $\tHoch^{II}(\cA)$, respectively, and the differentials on these non-reduced complexes descend to differentials on $\oHoch(\cA)$ and $\oHoch^{II}(A)$, which we also write as $b$. 
The definitions of $\oHoch(\cA)$ and $\oHoch^{II}(\cA)$ extend to the case where $\cA$ is a general cdg category by modding out by $\id_X \in
\Hom_{\cA}(X,X)$ for all $X$ in $\cA$.

\begin{prop}  \label{normalized-Hoch}
If $\cA$ is a curved differential $\G$-graded category, the canonical surjection
$$
\tHoch^{II}(\cA) \onto \oHoch^{II}(\cA)
$$
is a quasi-isomorphism.
\end{prop}

\begin{proof}
We give a proof in the case where $\cA$ is a cdga $(A,d,h)$; the general case differs only in its notational complexity.
For $j \in \Z$, let $G^j \subseteq \tHoch^{II}(\cA)$ denote the $k$-vector space comprised of cochains $(\a_0, \a_1, \dots) \in \prod_{n \ge 0} A \otimes_k (\S A)^{\otimes n}$ such
that each $\a_n$ is of the form $\sum_{i = 1}^{l_n} a_{0,i}[a_{1,i} | \dots | a_{n, i}]$, where $|a_{0,i}| + \cdots + |a_{n,i}| \ge j$ for each $i$. Notice that $b(G^j) \subseteq G^j$, and we have a filtration
$$
\cdots \subseteq G^{j+1}  \subseteq G^{j} \subseteq \cdots \subseteq (\tHoch^{II}(\cA), b)
$$
of complexes. Let $\pi$ denote the canonical surjection $\tHoch^{II}(\cA) \onto \oHoch^{II}(\cA)$,
and define $\oG^j = \pi(G^j)$, so that $\oG^{\bu}$ is a filtration of $(\oHoch^{II}(\cA), b)$
and $\pi$ induces a map of filtrations $G^{\bullet} \to \oG^{\bullet}$.

We have
$\colim(\cdots \into G^{j+1} \into G^j \into \cdots ) = \tHoch^{II}(\cA)$, and similarly for the filtration of $\oHoch^{II}(\cA)$. Fix $j \in \Z$; we need only show $G^j
\to \oG^j$ is a quasi-isomorphism for all $j$. We have that 
$G^j$ is the inverse limit of the tower 
$$
 \cdots \onto G^j / G^{j +2} \onto  G^j / G^{j +1},
$$
and $\oG^j$ is the inverse limit of the tower 
$$
 \cdots \onto \oG^j / \oG^{j +2} \onto  \oG^j / \oG^{j +1}.
$$
The evident map joining these two towers induces the map $\pi: G^j \to \oG^j$, and so it suffices to
prove $G^j/G^{j+m} \to  \oG^j/\oG^{j+m}$ is a quasi-isomorphism for all $m \geq 1$. 
For each such $m$, we have a morphism of short exact sequences
$$
\xymatrix{
 0 \ar[r] & G^{j + m}/G^{j+m + 1} \ar[d] \ar[r] & G^j/ G^{j + m + 1}  \ar[r] \ar[d] & G^j/G^{j + m } \ar[r] \ar[d] & 0 \\
0 \ar[r] & \oG^{j + m}/\oG^{j+m+1}  \ar[r] & \oG^j/ \oG^{j+m + 1}  \ar[r] & \oG^j/\oG^{j + m } \ar[r]  & 0 \\
}
$$
induced by $\pi$. Thus, by induction on $m$, it suffices to show that 
$$  
\pi: G^{l}/G^{l+1} \to \oG^{l}/\oG^{l+1} 
$$
is a quasi-isomorphism for all $l \in \Z$. This allows us to assume $d = 0 = h$ and that 
$A$ is concentrated in degree 0 (and so, in particular, $\tHoch^{II}(\cA) = \tHoch(\cA)$). But in this case, the result is well-known; see, for instance, \cite[Proposition 1.6.5]{Loday}.
\end{proof}

\begin{rem} The analogous statement for the ordinary Hochschild complex cannot be proven in the above manner, because the analogue of $G^j$ is not 
  the inverse limit of the analogous tower. 
\end{rem}

We next endow each of $\tHoch(\cA)$, $\tHoch^{II}(\cA)$, $\oHoch(\cA)$ and $\oHoch^{II}(\cA)$
with the structure of mixed complexes by introducing the Connes $B$ operator. The
following construction generalizes the standard one for ordinary algebras; see,
for instance, \cite[Section 2.1.7]{Loday}.

As before, we start with the case where $\cA = (A,d,h)$ is a  cdga. 
For each $n \geq 0$, define
\begin{equation} \label{E330}
\tau_{n+1}: A \otimes_k (\S A)^{\otimes n} \to A \otimes_k (\S A)^{\otimes n} 
\end{equation}
by
$$
\tau_{n+1} (a_0 [ a_1 | \cdots | a_n]) = (-1)^{(|a_0| - 1)(|a_1| + \cdots + |a_n| - n)} a_1 [ a_2 | \cdots | a_n | a_0],
$$
and define $s_0: \tHoch(\cA) \to \tHoch(\cA)$ by
$$
s_0 (a_0 [ a_1 | \cdots | a_n]) = 1 [a_0 | a_1 | \cdots | a_n].
$$
The \emph{Connes $B$ operator} is the degree $-1$ endomorphism $B$ of $\tHoch(\cA)$ whose restriction to $A \otimes (\S A)^{\otimes n}$ is given by 
$$
B|_{A \otimes (\S A)^{\otimes n}} = (1 - \tau^{-1}_{n+2}) \circ s_0 \circ \sum_{l=0}^n \tau_{n+1}^l.
$$

The Connes $B$ operator induces an endomorphism on the quotient $\oHoch(\cA)$ of $\tHoch(\cA)$, and we write the induced map also 
 as $B$. It is given by the simpler formula 
$$
B|_{A \otimes ( \S \overline{A})^{\otimes n}} = s_0 \circ \sum_{l=0}^n \tau_{n+1}^l,
$$
since $\tau^{-1}_{n+2} \circ s_0 \circ \sum_{l=0}^n \tau_{n+1}^l$ is the zero map on $\oHoch(\cA)$.
In detail:
\begin{equation} \label{Bformula}
B(a_0[\oa_1| \cdots |\oa_n]) = \sum_{l=0}^n (-1)^{(|a_l| + \cdots +|a_n| - (n-l+1)) (|a_0| + \cdots +|a_{l-1}| - l)   } 
1[\oa_l | \cdots |\oa_n| \oa_0 | \cdots | \oa_{l-1}].
\end{equation}
For example, if each $a_i$ has even degree, then 
$$
B(a_0[\oa_1| \cdots |\oa_n]) = \sum_{l=0}^n (-1)^{nl} 1[\oa_l | \cdots |\oa_n| \oa_0 | \cdots | \oa_{l-1}].
$$

On $\oHoch(\cA)$, we can equivalently define $B$ as follows. 
For each $m \geq 0$, define
$$
\s_{m}: (\S A)^{\otimes m} \to (\S A)^{\otimes m} 
$$
by
$$
\begin{aligned}
\s_{m} ([ a_1 | \cdots | a_m]) & = (-1)^{|sa_1|(|s a_2| + \cdots + |sa_m|)} [ a_2 | \cdots | a_m | a_1] \\
& = (-1)^{(|a_1|-1)(|a_2| + \cdots + |a_m| - (m-1))} [ a_2 | \cdots | a_m | a_1]. \\
\end{aligned}
$$
Then
$$
B|_{A \otimes ( \S \overline{A})^{\otimes n}} = \sum_{l=0}^n (\id_A \otimes \s_{n+1}^l) \circ s_0.
$$

The Connes $B$ operator is continuous for the topology on 
$\tHoch(\cA)$ given by the filtration $F^0 \supseteq F^1 \supseteq \cdots$ defined in \eqref{E721}, and we also write $B$ for the induced map on the completion
$\tHoch^{II}(A)$, and on $\oHoch^{II}(A)$ as well.

The definition of the operator $B$  extends to cdg
categories in an evident manner.
We record the following (cf. \cite[Section 3.1]{efimov}):

\begin{prop} If $\cA$ is a curved differential $\G$-graded category, we have $b B + B  b = 0$ and $B^2 = 0$, so that 
each of $(\tHoch(\cA), b, B)$, $(\oHoch(\cA), b, B)$, $(\tHoch^{II}(\cA), b, B)$ and $(\oHoch^{II}(\cA), b, B)$ is a  
$\G$-graded mixed complex over $k$. 
\end{prop}

\begin{defn} For a   cdg category $\cA$, define mixed complexes 
$$
\begin{aligned}
\tMC(\cA) & := (\tHoch(\cA), b, B) \\
\tMC^{II}(\cA) & :=(\tHoch^{II}(\cA), b, B), \\ 
\oMC(\cA) & :=(\oHoch(\cA), b, B), \text{ and} \\
\oMC^{II}(\cA) & :=(\oHoch^{II}(A), b, B).
\end{aligned}
$$ 
\end{defn}

The following is an immediate consequence of Proposition \ref{normalized-Hoch}:

\begin{prop} 
\label{normalized}
If $\cA$ is a   curved differential $\G$-graded category, the canonical morphism
$$
\alpha: \tMC^{II}(\cA) \to \oMC^{II}(\cA)
$$
of mixed complexes is a quasi-isomorphism.
\end{prop}

We now discuss functoriality.
It is established in \cite[Section 2.4]{PP} that, for each   functor (resp. strict   functor) 
$F: \cA \to \cA'$ of   cdg categories, there exists a morphism of complexes
$F_*: \tHoch^{II}(\cA) \to \tHoch^{II}(\cA')$ (resp., $F_*: \tHoch(\cA) \to \tHoch(\cA')$). Thus, 
$\tHoch^{II}(-)$ (resp., $\tHoch(-)$) is a functor from the category of   curved differential $\G$-graded categories, with morphisms given by   cdg
functors (respectively, strict   cdg functors), to the category of complexes of $k$-vector 
spaces.  

We recall the formula for $F_*$ in the case where $\cA$ and $\cA'$ are   cdga's; the general case is similar. Recall that a morphism of cdga's from
$(A,d,h)$ to $(A', d',h')$ is a pair $(\rho, 
\b)$, where $\rho: A \to A'$ is a morphism of graded algebras and $\b \in 
(A')^1$, satisying $\rho \circ d - d' \circ \rho = [\beta, \rho]$ and $\rho(h) = h' + d'(\beta) + \beta^2.$ The induced morphism
$$
(\rho, \b)_*: \tHoch^{II}(A,d,h) \to \tHoch^{II}(A',d',h') 
$$
sends a class of the form $a_0 [a_1| \dots| a_n]$ to 
\begin{equation} \label{E71b}
\sum_{i_0, \dots, i_n \geq 0} (-1)^{i_0 + \cdots + i_n} \rho(a_0) [\underbrace{\b | \cdots | \b}_{i_0 \text{ copies}} | \rho(a_1) |\underbrace{\b | \cdots | \b}_{i_1 \text{ copies}}| \rho(a_2) | \cdots | \rho(a_n) | \underbrace{\b | \cdots | \b}_{i_n \text{ copies}}],
\end{equation}
where the sum ranges over all ($n+1$)-tuples of non-negative integers. The map $(\rho, \b)_*$ extends to infinite sums in the evident way. Henceforth, for brevity,
we will write $(\rho, \b)_*(a_0 [a_1| \dots| a_n])$ in the following way: 
$$
\sum (-1)^{i_0 + \cdots + i_n} \rho(a_0) [\b^{i_0} | \rho(a_1) |\b^{i_1}| \rho(a_2) | \cdots | \rho(a_n) | \b^{i_n}].
$$
If $\b = 0$ (i.e., when the morphism is strict), the map $(\rho, 0)_*$ restricts to a map $\tHoch(A, d, h) \to \tHoch(A', d', h')$, 
but, since the sum occurring above is an infinite one when $\b \ne 0$, $\tHoch(-)$ is not functorial for arbitrary functors of cdg categories.

\begin{rem}
The sign appearing in the formula for $(\rho, \b)_*$ corresponds to the sign in line (15) of \cite[Section 2.4]{PP}  via the isomorphism discussed in Remark~\ref{PPsigns}.
\end{rem}

\begin{prop}
\label{functoriality}
If $F: \cA \to \cA'$ is a   cdg functor (resp. strict   cdg functor), then there is a morphism of mixed complexes 
$$
F_*:\oMC^{II}(\cA) \to \oMC^{II}(\cA') \text{ } (\text{respectively, } F_*: \oMC(\cA) \to \oMC(\cA'))
$$
determined by the formula \eqref{E71b}.
\end{prop}

\begin{proof}
By results of Polishchuk-Positselski quoted above, we need only show that the maps 
$F_*:\oMC^{II}(\cA) \to \oMC^{II}(\cA')$ and $F_*: \oMC(\cA) \to \oMC(\cA')$ preserve Connes $B$ operators. It suffices to prove the statement involving $\oMC^{II}( - )$; the statement involving $\oMC(-)$ is an immediate consequence. Moreover, we will assume $\cA$ and $\cA'$ are cdga's; the general case is similar.

Given a morphism 
$$(\rho, \b): (A, d, h) \to (A', d', h')$$ 
of   cdga's, let us verify that 
$$(\rho, \b)_*: \oHoch^{II}(A, d, h) \to \oHoch^{II}(A', d', h')$$
commutes with the Connes $B$ operators. Throughout this proof, we shall use the following notation: given $a_0 [\oa_1 | \cdots | \oa_n]$,
\begin{itemize}
\item write $a_i'$ for $\rho(a_i)$, and
\item  $\e_l = (-1)^{(|sa_l| + \cdots + |sa_n|) (|sa_0| + \cdots + |sa_{l-1}|)   }
= (-1)^{(|a_l| + \cdots + |a_n| - (n-l+1)) (|a_0| + \cdots + |a_{l-1}| - l)   }.$
\end{itemize}

We have
$$
(B \circ (\rho, \beta)_*)(a_0[\overline{a_1} | \cdots | \overline{a_n}] )= \sum_{N = 0}^{\infty} \sum_{|\vec{i}| = N} (-1)^N B(a_0'[\overline{\beta}^{i_0} | \overline{a_1'} | \overline{\beta}^{i_1} |  \cdots | \overline{a_n'} | \overline{\beta}^{i_n}]),
$$
where the inner sum ranges over all $\vec{i} = (i_0, \dots, i_n)$ such that $|\vec{i}| :=i_0 + \cdots + i_n = N$. Evaluating $B$, we obtain
$$
(B \circ (\rho, \beta)_*)(a_0[\overline{a_1} | \cdots | \overline{a_n}] ) = \sum_{N = 0}^{\infty} \sum_{|\vec{i}| = N} \sum_{l = 0}^{n + N} (-1)^N   \sigma^l_{n + N + 1}([\overline{a_0'}|\overline{\beta}^{i_0} | \overline{a_1'} | \overline{\beta}^{i_1} |  \cdots | \overline{a_n'} | \overline{\beta}^{i_n}]),
$$
where $\sigma_m: (\Sigma A)^{\otimes m} \to (\Sigma A)^{\otimes m}$ is as defined above. On the other hand, 

$$
((\rho, \beta)_* \circ B)(a_0[\overline{a_1} | \cdots | \overline{a_n}] ) = (\rho, \beta)_*(\sum_{l = 0}^n \e_l [\overline{a_l} | \cdots | \overline{a_n} | \overline{a_0} | \cdots | \overline{a_{l-1}}])
$$
$$
= \sum_{N = 0}^{\infty} \sum_{|\vec{j}| = N} \sum_{l = 0}^n (-1)^N\e_l[\overline{\beta}^{j_0} | \overline{a_l'} | \overline{\beta}^{j_{1}} |\cdots | \overline{a_n'} | \overline{\beta}^{j_{n-l + 1}} | \overline{a_0'} | \overline{\beta}^{j_{n-l + 2}}| \cdots | \overline{a_{l-1}'} | \overline{\beta}^{j_{n+1}}],
$$
where the middle sum ranges over all $\vec{j} = (j_0, \dots, j_{n+1})$ such that $|\vec{j}| := j_0 + \cdots + j_{n+1} = N$. Fix $N \ge 0$. It will suffice to prove that 
\begin{equation}
\label{fixedN1}
\sum_{|\vec{i}| = N} \sum_{l = 0}^{n + N}  1 \cdot \sigma^l_{n + N + 1}([\overline{a_0'}|\overline{\beta}^{i_0} | \overline{a_1'} | \overline{\beta}^{i_1} |  \cdots | \overline{a_n'} | \overline{\beta}^{i_n}])
\end{equation}
is equal to 
\begin{equation}
\label{fixedN2}
\sum_{|\vec{j}| = N} \sum_{l = 0}^n \e_l[\overline{\beta}^{j_0} | \overline{a_l'} | \overline{\beta}^{j_{1}} |\cdots | \overline{a_n'} | \overline{\beta}^{j_{n-l + 1}} | \overline{a_0'} | \overline{\beta}^{j_{n-l + 2}}| \cdots | \overline{a_{l-1}'} | \overline{\beta}^{j_{n+1}}].
\end{equation}

Since there are $\binom{N + m -1}{m-1}$ ways of adding $m$ nonnegative integers up to $N$, there are $(N + n + 1)\binom{N + n }{n}$ terms in~(\ref{fixedN1}) and $(n+1)\binom{N + n
  + 1}{n+1}$ terms in~(\ref{fixedN2}).
Using that  $(N + n + 1)\binom{N + n }{n} = (n+1)\binom{N + n + 1}{n+1}$, 
we need only associate to each summand of~(\ref{fixedN1}) a distinct summand of~(\ref{fixedN2}) to
which it is equal, and it is easy to check that the summand of~(\ref{fixedN2}) corresponding to the pair  
$$(\vec{j} = (j_0, \dots, j_{n+1}), l)$$
is equal to the summand of~(\ref{fixedN1}) corresponding to the pair
$$( \vec{i} = (j_{n-l+2}, \dots, j_n, j_0 + j_{n+1}, j_1, \dots, j_{n - l + 1}), j_{n-l + 2} + \cdots + j_{n+1} + l).$$
\end{proof}

\begin{rem}
Let $F: \cA \to \cA'$ be a   cdg functor. The induced map
$$
F_*: \tHoch^{II}(\cA) \to \tHoch^{II}(\cA')
$$
does \emph{not} necessarily commute with $B$ operators, and it therefore does not induce a morphism of mixed complexes $\tMC^{II}(\cA) \to \tMC^{II}(\cA')$. 
For example, take $\cA$ to be the curved algebra (with trivial curvature) $(A, 0)$, where $A$ is an exterior algebra on two degree $1$ variables $e_1, e_2$, and consider the morphism
$$(\id_A, e_1): (A, 0) \to (A, 0)$$
of curved algebras from Example~\ref{functorexample}. Consider $e_2$ as an element of $\tHoch^{II}(\cA)$. Then 
$$((\id_A, e_1) \circ B)(e_2) = (\id_A, e_1)(1[e_2] - e_2[1])$$
\begin{equation}
\label{LHS}
= \sum_{N \ge 0} \sum_{i_0 + i_1 = N} (-1)^N (1[e_1^{i_0} | e_2 | e_1^{i_1}] - e_2[e_1^{i_0} | 1 | e_1^{i_1}]  ) .\end{equation}

On the other hand,
$$
(B \circ (\id_A, e_1))(e_2) = \sum_{N \ge 0} (-1)^N B(e_2[e_1^N])
$$
\begin{equation}
\label{RHS}
= \sum_{N \ge 0} (-1)^N ((1 - \tau_{N + 2}^{-1}) \circ s_0 \circ \sum_{l = 0}^N \tau^l_{N + 1})(e_2[e_1^N]).
\end{equation}

The summand of~(\ref{LHS}) corresponding to $N = 1$ is 
$$ 
-  1[ e_2 | e_1] + e_2[  1 | e_1]  -  1[e_1 | e_2 ] + e_2[e_1 | 1 ],
$$
and the summand of~(\ref{RHS}) corresponding to $N = 1$ is
\begin{align*}
(( \tau_{3}^{-1} - 1) \circ s_0 \circ \sum_{l = 0}^1 \tau^l_{2})(e_2[e_1]) &= ( \tau_3^{-1} -1) (1[e_2 | e_1] + 1[e_1 | e_2]) \\
&= e_1[1|e_2] + e_2[1 | e_1] -1[e_2|e_1] - 1[e_1| e_2].
\end{align*}
Since these summands differ, $((\id_A, e_1) \circ B)(e_2) \ne (B \circ (\id_A, e_1))(e_2)$. Notice that these summands coincide upon passing to $\oHoch^{II}(A, 0)$, as expected.
 \end{rem}

A key feature of the functor $\oMC^{II}(-)$ is that it sends pseudo-equivalences to quasi-isomorphisms. The following statement is proven in \cite[p. 5326]{PP}:

\begin{thm}[Polishchuk-Positselski]
\label{pseudoqi}
If $F: \cA \to \cA'$ is a pseudo-equivalence of   cdg categories, then
$F_*: \oMC^{II}(\cA) \to \oMC^{II}(\cA')$ is a quasi-isomorphism.
\end{thm}

It is well-known that Hochschild homology of dg categories is Morita invariant. Combining Proposition \ref{PPpseudo} and Theorem \ref{pseudoqi}, we obtain a
version of Morita invariance for mixed Hochschild complexes of   cdg categories which we dub ``quasi-Morita invariance'':

\begin{prop}
\label{thmPP}
Let $\cA$ be a   cdg category. The canonical inclusion 
$$\Perf(\cA)^\op \into \qPerf(\cA)^\op$$
and the quasi-Yoneda embedding 
$$qY_{\cA}: \cA  \into \qPerf(\cA)^\op$$ 
defined in  \eqref{E73a} 
both induce quasi-isomorphisms of mixed Hochschild complexes of the second kind:
$$
\oMC^{II}(\Perf(\cA)^\op) \xra{\simeq} \oMC^{II}(\qPerf(\cA)^\op) \xla{\simeq} \oMC^{II}(\cA).
$$
\end{prop}

\subsection{Negative cyclic homology of cdg categories} \label{invariants}

Let $(M, b, B)$ be a mixed complex. The {\em negative cyclic complex} associated to $(M,b,B)$ is defined to be 
$(M[[u]], b + uB)$, where $u$ is a degree $2$ indeterminant, and $M[[u]]$ is the direct product totalization of the ($\G \times \Z$)-graded $k$-vector space
$$
M, 0, Mu, 0, Mu^2, \dots
$$
So, an element of $M[[u]]$ of degree $e$ is a
formal power series $\sum_i m_i u^i$ such that $m_i \in M$ has degree $e - 2i \in \G$ for each $i$. Since $b^2 = B^2 = [b,B] = 0$, $|b| = 1$ and $|B| = -1$,  
$(M[[u]], b + uB)$ is indeed
a $\G$-graded complex of $k$-vector spaces. 

\begin{ex} 
\label{field}
The negative cyclic complex of the mixed complex $(k, 0, 0)$ has trivial differential and is $k[[u]]$. 
An element of $k[[u]]$ of degree $e$ is a
formal power series $\sum_i \a_i u^i$ such that $\a_i \in k$ has degree $e - 2i \in \G$, for each $i$, and multiplication is given by the
usual rule for multiplying power series.  
If $\G = \Z$, then $k[[u]]$ is just the usual polynomial ring $k[u]$ with $|u| = 2$ (and so the notation $k[[u]]$ is misleading in this case). 
But if $\G = \Z/2$, then $k[[u]]$ is the ring of formal power series, concentrated in degree $0$. 
\end{ex}

The graded $k$-module $k[[u]]$ is a commutative $\G$-graded $k$-algebra, with polynomial/power series multiplication. For any mixed complex $(M,b,B)$, the complex
$(M[[u]], b + uB)$ is a dg-$k[[u]]$-module in an evident way.

We define $k((u))$ to be the $\G$-graded ring obtained from $k[[u]]$ by inverting the degree two element $u$. So, if $\G = \Z$, then $k((u))$ is the ring of
Laurent polynomials $k[u, u^{-1}]$ with $|u| = 2$, and if $\G = \Z/2$, then $k((u))$ is the ring of Laurent power series concentrated in degree $0$. 
The {\em periodic cyclic complex} of $(M,b,B)$ is the $\G$-graded dg-$k((u))$-module 
$$
(M((u)), b + uB) = 
(M[[u]], b + uB) \otimes_{k[[u]]} k((u))
$$
given by inverting $u$.

\begin{defn} Let $\cA$ be a cdg category. 
The \emph{(ordinary) negative cyclic complex} and the 
\emph{negative cyclic complex of the second kind} associated to $\cA$ are the negative cyclic complexes associated to  the mixed complexes $\tMC(\cA)$ and $\tMC^{II}(\cA)$,
respectively.
We write them as 
$$
\tHN(\cA) \and \tHN^{II}(\cA).
$$ 
Similarly, 
the \emph{(ordinary) periodic cyclic complex} and the \emph{periodic cyclic complex of the second kind} associated to $\cA$ are  the periodic cyclic complexes
associated to $\tMC(\cA)$ and $\tMC^{II}(\cA)$, 
and they are written as 
$$
\tHP(\cA) \and \tHP^{II}(\cA).
$$

Likewise, $\oHN(\cA)$, $\oHN^{II}(\cA)$, $\oHP(\cA)$ and $\oHP^{II}(\cA)$
are obtained from $\oMC(\cA)$ and $\oMC^{II}(\cA)$. 
By Proposition~\ref{normalized}, there are natural quasi-isomorphisms $\tHN(\cA) \xra{\sim} \oHN(\cA)$, etc. The negative and periodic cyclic homology groups, of both kinds, are defined as the homology of the complexes $\tHN(\cA)$, $\tHN^{II}(\cA)$, $\tHP(\cA)$ and $\tHP^{II}(\cA)$:
$$
\begin{aligned}
\HN_q(\cA) & := H^{-q} \tHN(\cA), \\
\HN^{II}_q(\cA) & := H^{-q} \tHN^{II}(\cA), \\
\HP_q(\cA) & := H^{-q} \tHP(\cA), \and \\
\HP^{II}_q(\cA) & := H^{-q} \tHP^{II}(\cA). \\
\end{aligned}
$$
(Since $H^{-q}(\tHN(\cA))$ is canonically isomorphic to $H^{-q}(\oHN(\cA))$, we do not introduce notation for the latter and its variants.)
\end{defn}

\subsection{The HKR theorem}
\label{HKRsection}

Recall that if $A$ is a classical $k$-algebra (i.e., an algebra concentrated in degree $0$), we say $A$  is {\em smooth} (over $k$) if it is isomophic to a 
commutative $k$-algebra of the form 
$B = k[x_1, \dots, x_n]/(f_1, \dots, f_m)$ such that the Jacobian matrix $\left(\partial f_i/ \partial x_j\right)$ has maximal rank at every point in $\Spec(B)$, and 
we say $A$ is  {\em essentially smooth} (over $k$)  
if it is isomorphic to $S^{-1}B$ for some smooth $k$-algebra $B$ and some multiplicatively closed subset $S$.

\begin{defn} 
\label{essentiallysmooth}
A $\G$-graded curved algebra $\cA = (A,h)$ (with trivial differential) is {\em essentially smooth over $k$} if 
\begin{itemize}
\item $A$ is commutative and concentrated in even degrees, and
\item upon forgetting the grading, $A$ is an essentially smooth $k$-algebra in the classical sense.
\end{itemize}
\end{defn}

\begin{ex} Suppose $A$ is an essentially smooth $k$-algebra in the classical sense, and let $h \in A$. Then 
$(A, h)$ is a $\Z/2$-graded essentially smooth curved algebra over $k$.
\end{ex}

\begin{ex} Suppose $Q$ is an essentially smooth $k$-algebra in the classical sense, let $f_1, \dots, f_c \in Q$, 
  and let $T_1, \dots, T_c$ be indeterminants of degree $2$. Then $(Q[T_1, \dots, T_c], \sum_i f_i T_i)$ is an essentially smooth curved $\Z$-graded algebra over $k$.
\end{ex}

Let $\cA = (A, h)$ be a $\G$-graded, essentially smooth curved $k$-algebra, and let $\O^\bullet_{A/k}$ denote the graded ring given by the exterior powers of the module of K\"ahler differentials of $A$ over $k$, $\G$-graded so that $a_0 da_1 \smsh
\cdots \smsh da_j$ has degree $\sum |a_i| - j$. We consider the mixed complex
$$
(\Omega^\bullet_{A/k}, dh, d),
$$
where $dh$ denotes the map given by {\em left} multiplication by the degree $1$ element $dh$, and $d$ is given by the de Rham differential (which, by our indexing convention, has degree $-1$).

The following is a generalization, to the setting of curved algebras, of the classical Hochschild-Kostant-Rosenberg Theorem.
See \cite[Proposition 3.14]{efimov} for a proof. 

\begin{thm}[``The HKR Theorem''] 
\label{HKR} 
Assume $\on{char}(k) = 0$. 
If $\cA = (A,h)$ is a $\G$-graded, essentially smooth curved $k$-algebra,
then there is a quasi-isomorphism of mixed complexes
$$
\epsilon: \oMC^{II}(\cA) \xra{\simeq} (\Omega^\bullet_{A/k}, dh, d)
$$
which sends elements of the form $a_0 [\overline{a_1} | \cdots |\overline{a_n}]$ to $\frac{1}{n!} a_0 da_1 \cdots da_n$. 
In particular, we have an isomorphism 
$$
\HN^{II}_q(\cA) \xrightarrow{\cong} H^{-q}(\Omega^\bullet_{A/k}[[u]], u d + dh)
$$
of graded $k[[u]]$-modules for all $q \in \Z$.
\end{thm}

\begin{rem}
Efimov's model for the Hochschild complex of the second kind associated to a curved algebra $\cA=(A, W)$, written as $\Hoch^{\prod}(A, W)$ in 
\cite[Proposition 3.14]{efimov}, is identical to our $\tHoch^{II}(\cA)$, except the differential is negated. Further, the Connes $B$ operators on both complexes
are the same. This is why the differential in the target of the HKR map in loc. cit. is the negative of ours, but the statements are otherwise the same.
\end{rem}

\section{Chern character maps for curved differential graded categories}
\label{chernsection}

The {\em Grothendieck group of a triangulated category} $T$ is the abelian group generated by isomorphism classes $[X]$ of objects of $T$ modulo the
relations $[X] = [X'] + [X'']$ whenever there is a  distinguished triangle $X' \to X \to X'' \to \S X$. We write this group as $K_0^\Delta(T)$. 
If $\cC$ is a dg category, the associated {\em homotopy category}, written $[\cC]$, is the ordinary category having the same objects as $\cC$ and with hom sets
given by $\Hom_{[\cC]}(X,Y) := H^0 \Hom_{\cC}(X, Y)$. When $\cC$ is a pretriangulated dg category, $[\cC]$ inherits a canonical triangulated structure; this occurs, for instance, when
$\cC = \Perf(\cA)$ for some cdg category $\cA$. See \cite[Section 4.5]{kelleron} for a discussion of pretriangulated dg categories.

\begin{defn} 
Let $\cA$ be a   curved differential $\G$-graded category. The \emph{Grothendieck group} of $\cA$, written $K_0(\cA)$, is 
the Grothendieck group of the (triangulated) homotopy category $[\Perf(\cA^\op)]$ of the dg category of perfect right $\cA$-modules:
$$
K_0(\cA) := K_0^\Delta([\Perf(\cA^\op)]).
$$
So, $K_0(\cA)$ is generated by isomorphism
classes of perfect right $\cA$-modules modulo relations coming from homotopy equivalences and distinguished triangles.
\end{defn}

In this section, we define and develop the (ordinary) Chern character map and the Chern character map of the second kind for   cdg categories. 
The former is a homomorphism of the
form 
$$
ch_{HN}: K_0(\cA) \to \HN_0(\cA),
$$
where $\cA$ is any   dg category (i.e., a cdg category with trivial curvatures). The latter is a homomorphism of the form
$$
ch^{II}_{HN}: K_0(\cA) \to \HN^{II}_0(\cA),
$$
and it is defined for an arbitrary   cdg category.
For a   dg category $\cA$,
the two are related by the equation $ch^{II}_{HN} = \can \circ ch_{HN}$, where $\can$ is the canonical map from $\HN_0$ to
$\HN_0^{II}$. 

\begin{rem} If $\cA$ has non-trivial curvature, then the map $ch_{HN}$ is undefined, and no such factorization exists. This is as expected:
if, for example, $\cA = (A, d, h)$ is a cdga with $h \ne 0$, then $\HN_0(\cA) = 0$, but
the map $ch^{II}_{HN}$ is often non-trivial. 
\end{rem}

\subsection{The Chern character map for dg categories}
\label{cherndgcat}

We first review the construction of the ordinary Chern character map 
$$
ch_{HN}: K_0(\cA) \to \HN_0(\cA)
$$
for a dg category $\cA$, as defined, for instance, in \cite{keller}.  Merely knowing that such a map exists and is natural is enough to
determine it, as we now explain. 

There are isomorphisms
$$
\Z \xra{\cong} K_0(k)
$$
and 
$$
k[[u]]_0 \xra{\cong} \HN_0(k)
$$
given by the ring map sending $1$ to $[k]$ and the $k[[u]]_0$-linear map sending $1$ to the class $\g$ represented by the constant power series $1 \in \oHoch(k)[[u]]$, respectively. Recall that if $\G = \Z$, then $k[[u]]_0 = k$, and if $\G = \Z/2$ then $k[[u]]_0 =
k[[u]]$, the ring of formal power series. Under these isomorphisms, the Chern character 
\begin{equation}
\label{chernfield}
ch_{HN}: K_0(k) \to \HN_0(k)
\end{equation}
is the unique ring map $\Z \to k[[u]]_0$. 

\begin{rem}
\label{canonicalclass}
The element $1$ is a cycle in the complex $(\oHoch^{II}(k)[[u]], b + uB)$, but not in $(\Hoch^{II}(k)[[u]], b + uB)$ (unless $\chr(k) = 2$). One easily checks that
$$\sum_{i = 0}^{\infty} c_i 1[\underbrace{1 | \cdots | 1}_{2i \text{ copies}}]u^i \in (\Hoch^{II}(k)[[u]], b + uB),$$
where $c_i := (-1)^i 2^i \prod_{j = 0}^{i-1} (2j + 1)$, is a cycle; the projection 
$$
(\Hoch^{II}(k)[[u]], b + uB) \onto (\oHoch^{II}(k)[[u]], b + uB)
$$
sends this cycle to $1$.
\end{rem}

Both $K_0( - )$ and $\HN_0( - )$ are \emph{Morita invariant}; that is, the Yoneda map
$$
Y_\cA: \cA \to \Perf(\cA^\op), X \mapsto \Hom_{\cA}(-, X).
$$
induces isomorphisms upon applying $K_0( - )$ and $\HN_0( - )$.
Let $X$ be an object of $\Perf(\cA^\op)$. There is a canonical map $\alpha_X: k \to \End(X) := \End_{\Perf(\cA^\op)}(X)$ sending $1$ to $\id_X$ and an inclusion
$\inc: \End_{\Perf(\cA^\op)} (X) \into \Perf(\cA^\op)$. 
The naturality of $ch$ yields the commutative diagram
\begin{equation} \label{E630b}
\xymatrix{
K_0(k) \ar[r]^-{(\alpha_X)_*} \ar[d]^-{ch_{HN}} & K_0(\End(X)) \ar[r]^{\inc_*} \ar[d]^-{ch_{HN}} & K_0(\Perf(\cA^\op)) \ar[d]^-{ch_{HN}} & K_0(\cA) \ar[l]^-{\cong}_-{(Y_\cA)_*} \ar[d]^-{ch_{HN}}\\
\HN_0(k) \ar[r]^-{(\alpha_X)_*}     & \HN_0(\End(X)) \ar[r]^{\inc_*}  & \HN_0(\Perf(\cA^\op)  & \HN_0(\cA) \ar[l]^-{\cong}_-{(Y_\cA)_*}.\\
}
\end{equation}
From this diagram we deduce:

\begin{prop} If $\cA$ is a  differential $\G$-graded category, and $X \in \Perf(\cA^\op)$ is a perfect right $\cA$-module, 
then $ch_{HN}(X)$ is the image of $\g \in  \HN_0(k)$ under $(Y_\cA)_*^{-1} \circ \inc_* \circ (\alpha_X)_*$.
\end{prop}

\subsection{The Chern character map of the second kind}

We now wish to define a Chern character map  of the second kind
$$
ch^{II}_{HN}: K_0(\cA) \to \HN^{II}_0(\cA)
$$
for a   cdg category $\cA$. We cannot proceed exactly as in the previous section, since there is no Yoneda embedding $\cA \into \Perf(\cA^\op)$
(Example~\ref{representable}). Instead, we use the quasi-Yoneda embedding  
$$
qY_{\cA}: \cA \into \qPerf(\cA^\op), X \mapsto \Hom_{\cA}(-, X),
$$
and the quasi-Morita invariance of mixed Hochschild complexes of the second kind (Proposition~\ref{thmPP}).

Recall that $K_0(\cA)$ is defined to be $K_0^\Delta([\Perf(\cA^\op)])$. The Yoneda embedding
$$
Y = Y_{\Perf(\cA^\op)} : \Perf(\cA^\op) \into \Perf(\Perf(\cA^\op)^\op)
$$
induces a triangulated functor on homotopy categories
$$
[Y]: [\Perf(\cA^\op)] \xra{} [\Perf(\Perf(\cA^\op)^\op)];
$$
the target is the idempotent completion of the source \cite[Section 4.6]{kelleron}.  Note that $[\Perf(\cA^\op)]$ need not be idempotent complete; for instance, take $\cA$ to be
the $\Z/2$-graded curved algebra $(\C[x,y]_{(x,y)}, -y^2 + x^2(x+1))$, where $\C[x,y]_{(x,y)}$ is concentrated in degree $0$. Then $\Perf(\cA^{\op})$ coincides with the matrix
factorization category $\MF(\C[x,y]_{(x,y)}, y^2 - x^2(x+1))$, and, as shown in \cite[Section 2.5]{thesis}, the homotopy category of this matrix factorization category is
not idempotent complete.  

The triangulated functor $[Y]$ induces a canonical map
$$
K_0^\Delta([Y]): K_0(\cA) \to K_0(\Perf(\cA^{op})).
$$

\begin{defn} For a   cdg category $\cA$, the {\em Chern character map of the second kind} is the homomorphism
$$
ch_{HN}^{II}: K_0(\cA) \to \HN^{II}_0(\cA)
$$
given by the composition 
$$
\begin{aligned}
K_0(\cA) & \xra{K_0^\Delta([Y])} K_0(\Perf(\cA^\op))  \xra{ch_{HN}}  \HN_0(\Perf(\cA^\op)) \\
& 
\xra{\can}  \HN^{II}_0(\Perf(\cA^\op)) \to \HN^{II}_0(\qPerf(\cA^\op)) \xra{(qY_{\cA})_*^{-1}}  \HN^{II}_0(\cA),
\end{aligned}
$$
where the penultimate map is induced by the inclusion $\Perf(\cA^\op) \subseteq \qPerf(\cA^\op)$, and the last map is the inverse of the isomorphism induced by
the quasi-Yoneda embedding $qY_{\cA}$ (see Proposition~\ref{thmPP}).
\end{defn}

\begin{rem} If $\cA$ is a dg category, then, since the quasi-Yoneda embedding lands in $\Perf(\cA)$, a diagram chase shows that
$$
ch_{HN}^{II} = \can \circ ch_{HN},
$$
where $\can: \HN_0(\cA) \to \HN^{II}_0(\cA)$ is the canonical map. 
\end{rem}

\subsection{The Chern character map of the second kind for cdg modules over a cdga}
\label{cherncdgmodules}

We now specialize our general construction to the case of a   cdga
$\cA = (A,d,h)$. 
In the previous section, we established a Chern character map of the second kind
$$
ch_{HN}^{II}: K_0(\cA) \to \HN^{II}_0(\cA).
$$
In this section, we extend this construction slightly.

Given a  perfect right quasi-module $\cP = (P, \d_P) \in \qPerf(\cA^\op)$, recall that $\End(\cP) := \End_{\qPerf(\cA^\op)}(\cP)$ 
is the cdga $(\End_A(P), [\d_P,-], \d_P^2 + \rho_h)$ where
$\rho_h$ is right multiplication by $h$; see Remark \ref{rem72c}.

\begin{defn} \label{def41}
For a perfect right quasi-module $\cP = (P, \d_P) \in \qPerf(\cA^\op)$,
define the map 
$$
ch^{II}_{\cP}: \HN_0^{II}(\End(\cP)) \to  \HN^{II}_0(\cA)
$$
to be the composition of the maps induced by the inclusion of $\End(\cP)$
into $\qPerf(\cA^\op)$ and the inverse of the quasi-Yoneda isomorphism: 
$$
\HN_0^{II}(\End(\cP)) \to  \HN^{II}_0(\qPerf(\cA^\op)) \xra{\cong} \HN^{II}_0(\cA).
$$
\end{defn}

If $\cP \in \Perf(\cA^\op)$, so that $\End(\cP)$ has trivial curvature, then there is a distinguished class
$$
\g_P \in \HN_0^{II}(\End(\cP))
$$
represented by the constant power series $\id_P \in \oHoch^{II}(\End(\cP))[[u]]$. The reason we use the notation $\g_P$ rather than $\id_P$ for this distinguished class is that $\id_P$ is not a cycle in $(\Hoch^{II}(\End(\cP))[[u]], b + uB)$ (see Remark \ref{canonicalclass}).

The following is immediate from the definitions:

\begin{prop} For any cdga $\cA$ and any $\cP \in \Perf(\cA^\op)$, 
we have
$$
ch_{HN}^{II}([\cP]) = ch^{II}_{\cP}(\g_P).
$$
\end{prop}

For $\cP \in \qPerf(\cA^\op)$, recall from Example~\ref{perfpseudo} the notation $\cP^\nat:=(P, 0) \in \qPerf(\cA^\op)$. We have an isomorphism of cdga's
$$
(\id, \d_P): \End(\cP) \to \End(\cP^\nat)
$$
with inverse given by
$$
(\id, -\d_P): \End(\cP^\nat) \to \End(\cP).
$$

The following proposition, which relates the Chern character maps for $\cP$ and $\cP^\nat$, will play a crucial role later on:

\begin{prop}
\label{comparison}
For a cdga $\cA$ and $\cP \in \qPerf(\cA^\op)$, we have 
$$
ch^{II}_{\cP^\nat} \circ (\id, \d_P)_* = ch^{II}_{\cP}.
$$
In particular, if $\cP \in \Perf(\cA^\op)$, we have
$$
ch_{HN}^{II}(\cP) = ch^{II}_{\cP^\nat} \left(\sum_j   (-1)^j [ \underbrace{\ov{\d_P}| \cdots | \ov{\d_P}}_{j \text{ copies}}]\right) \in H_0(\oHoch^{II}(\cA)[[u]], b + u B) 
=  \HN^{II}_0(\cA).
$$ 
\end{prop}

\begin{proof} 
Let $\{\cP, \cP^\nat\}$ denote the full subcategory of $\qPerf(\cA^\op)$ consisting of just the two
indicated objects, and let
$$
\iota: \{\cP^\nat\}\to \{\cP, \cP^\nat\} \and
\iota': \{\cP\} \to \{\cP, \cP^\nat\}
$$
denote the inclusion functors. The functor $\iota$ admits a left inverse; that is, there is a cdg functor 
$$
G: \{\cP, \cP^\nat\} \to \{\cP^\nat\}
$$
such that $G \circ \iota = \id$.
The functor 
$G$ is defined in the only way possible on objects; on morphisms, the map
$\End(\cP)\to \End(\cP^\nat)$ is $(\id, \d_P)$, the map
$\End(\cP^\nat) \to \End(\cP^\nat)$ is $(\id, 0)$, and both maps $\Hom(\cP, \cP^\nat)  \to \End(\cP^\nat)$ and 
$\Hom(\cP^\nat, \cP)  \to \End(\cP^\nat)$ are the identity. Observe that $G \circ \iota' = (\id, \d_P)$.

Consider the diagram of graded $k[[u]]$-modules
$$
\xymatrix{
\HN_*^{II}(\{\cP\}) \ar[rr]^-{(\id, \d_P)_*} \ar[rd]^-{\iota'_*} \ar[rdd]_{} &    & \HN_*^{II}(\{\cP^\nat\}) \ar[ddl]^{} \ar[dl]_-{\iota_*} \\
& \HN_*^{II}(\{\cP, \cP^\nat\})  \ar[d]^{} \\
& \HN_*^{II}(\qPerf(\cA^\op)) \\
}
$$
in which each map, except for the horizontal one,
is induced by the evident inclusion of cdg categories. 

To prove $ch_{\cP^\nat} \circ (\id, \d_P)_* = ch_{\cP}$, it suffices to show that the exterior triangle commutes. Clearly the bottom two interior triangles commute, so we need only show
$\iota'_* =  \iota_* \circ (\id, \d_P)_* $. By Example~\ref{perfpseudo} and Proposition~\ref{functoriality}, $\iota_*$
and $\iota'_*$ are isomorphisms. Thus, we have $G_* = (\id, \d_P)_* \circ (\iota'_*)^{-1}$, and composing both sides
with $\iota_*$ on the left yields the result. 

Finally, by \eqref{E71b}, we have
$$
(\id, \d_P)_*(\g_P) = \sum_j (-1)^j[ \underbrace{\ov{\d_P}| \cdots | \ov{\d_P}}_{j \text{ copies}}]).  \qedhere
$$
\end{proof}

\section{The main theorem}
\label{themaintheorem}

Assume $\on{char}(k) = 0$, and let $\cA = (A,h)$ be an essentially smooth curved $k$-algebra (Definition~\ref{essentiallysmooth}). Upon composing the Chern character map
$$
ch_{HN}^{II}: K_0(\cA) \to \HN_0^{II}(\cA)
$$
with the HKR isomorphism
$$
\e: \HN_0^{II}(\cA) \xra{\cong} H_0(\O^\bullet_{A/k}[[u]], ud + dh),
$$
we get a map 
$$
ch^{II}_{HKR}: K_0(\cA) \to H_0(\O^\bullet_{A/k}[[u]], ud + dh).
$$
The main theorem of this paper provides a Chern-Weil-type formula for the class $ch^{II}_{HKR}(\cP)$ associated to a perfect right module $(P, \d_P) \in
\Perf(\cA^\op)$.

\subsection{Connections, curvature, and the trace map}
Before giving the precise statement of the main theorem, we establish some terminology.

\begin{defn} If $\cP = (P, \d_P) \in \qPerf(\cA)$, a \emph{connection} on $\cP$ is a degree $-1$ 
morphism 
$$
\n: P \to P  \otimes_A  \Omega^1_{A/k}
$$
of graded $k$-vector spaces such that $\n(pa) =   \n(p)a + (-1)^{|p|}p \otimes d(a)$ for all $a \in A$ and $p \in P$, where $d$ denotes the de Rham differential on
$\Omega^\bullet_{A/k}$. That is, a connection on $\cP$ is just a connection on $P$ which respects the grading.
\end{defn}

\begin{rem}
Our grading convention for $\Omega^\bullet_{A/k}$ is the same as the one used in Section~\ref{HKRsection}, so $a_0da_1 \wedge \cdots \wedge da_j$ has degree $\sum |a_j| - j$.
\end{rem}

A connection $\nabla$ on $(P, \d_P) \in \qPerf(\cA)$ extends uniquely to a degree $-1$ map
$$
\tn :  P \otimes_A \Omega^\bullet_{A/k}   \to  P \otimes_A \Omega^\bullet_{A/k} 
$$
of graded $k$-vector spaces such that $\tn( p \otimes \o) =  \tn(p) \o + (-1)^{|p|}p \otimes d \o$.
Let 
$$
\n^2: P \to  P \otimes_A \Omega^2_{A/k}
$$ 
denote the restriction of $(\tn)^2$ to $P$. Both $\n^2$ and the map
$$
[\n, \d_P]: P \to  P \otimes_A \Omega^1_{A/k}
$$
are $A$-linear maps, of degrees $-2$ and $0$, respectively.  
Since $A$ is commutative and $P$ is projective, there is a canonical isomorphism
$$
\Hom_{A}(P,  P \otimes_A \Omega^\bullet_{A/k}) \xra{\cong}  \End_{A}(P) \otimes_A \Omega^\bullet_{A/k},
$$
and thus we may identify $\n^2$ and $[\n, \d_P]$ with elements of $\End_{A}(P) \otimes_A \Omega^\bullet_{A/k}$.  

\begin{defn} The {\em curvature} of a connection $\n$ on $(P, \d_P) \in \qPerf(\cA^{op})$ is the class
$$
R := \n^2 u + [\n ,\d_P]  \in  \End_A(P)\otimes_A \Omega^\bullet_{A/k}[[u]].
$$
\end{defn}

\begin{rem}
Recall that, classically, the curvature of a connection $\n$ on an ordinary algebraic vector bundle is given by just $\n^2$. As discussed in the introduction, the presence of the extra term $[\n, \d_P]$ is motivated by Quillen's formula for the Chern character of a relative topological $K$-theory class in \cite{quillen}.
\end{rem}

For a projective $A$-module $P$, we have a canonical isomorphism
$$
P \otimes_A P^* \xra{\cong} \End_A(P)
$$
given by $p \otimes \g \mapsto (x \mapsto p\g(x))$. We define the trace map
$$
\tr: \End_A(P) \to A
$$
to be the map corresponding, via this isomorphism, to
$$\on{ev}: P \otimes_A P^* \to A,$$
where $\on{ev}(p \otimes \gamma) = (-1)^{|p||\g|} \gamma(p).$
We extend $\tr$ to a map
$$
\tr:  \End_{A}(P) \otimes_A  \Omega^\bullet_{A/k}[[u]] \to \Omega^\bullet_{A/k}[[u]]
$$ 
by extension of scalars: $\tr(\a \otimes \o) = \tr(\a) \o$. 

\begin{rem}
\label{tr}
If $P$ is a graded free $A$-module with basis $e_1, \dots, e_r$, a homogeneous element of $\End_A(P)$ of degree $m$ may be identified with an $r \times r$
matrix $(a_{i,j})$ of homogeneous elements with $|a_{i,j}| = |e_i| - |e_j| + m$. In this case, $\tr$ is the $\Omega^\bullet_{A/k}[[u]]$-linear map which sends
$(a_{i,j})  \otimes 1 $ to $\sum_i (-1)^{|e_i|} a_{i,i} \in \O^{\bullet}_{A/k}[[u]]$. 
When $P$ is not necessarily free, $\tr$ may be described locally in the above manner. Using this observation, we conclude that
\begin{enumerate}
\item $\tr( \a \otimes \o) = 0$ if $\a \in \End_A(P)$ has odd degree (since $A$ has no non-zero odd degree elements), 
and
\item $\tr \circ [ - , -] = 0$, where $[- , -]$ denotes the graded commutator in the algebra $\End_A(P) \otimes_A  \Omega^\bullet_{A/k}[[u]] $.
\end{enumerate}
\end{rem}

The following lemma is adapted directly from classical Chern-Weil theory; see, for instance, \cite[Lemma 8.1.5]{Loday}.

\begin{lem} \label{lem326b} 
Let $\cA = (A, h)$ be an essentially smooth curved $k$-algebra, 
let $\cP = (P, 0)$ be an object in $\qPerf(\cA)$, and let $\n$ be a connection on $\cP$. We have
$$
\tr \circ [\n, -] = d \circ \tr,
$$
where $\tr: \End_A(P) \otimes_A \O^\bullet_{A/k} \to \O^\bullet_{A/k}$ is the trace map, and $d$ is the de Rham differential.
\end{lem}

In fact, the curvature plays no role here, so this lemma is almost identical to the classical version. But, since there is a $\G$-grading to keep track of, our
statement is slightly more general, 
and so we provide a proof.

\begin{proof} Localizing at a homogeneous prime ideal of $A$, we may assume $P$ is graded free \cite[Proposition 1.15(d)]{BH}. 
Let $n$ denote the rank of $P$. Identifying $\End_A(P) \otimes_A \O^\bullet_{A/k}$ with $\Mat_{n \times n}(\O^\bullet_{A/k})$, the connection $\n$ is given by 
$$
\n(v) = dv + \theta v
$$
for some matrix of one-forms $\theta \in \Mat_{n \times n} (\O^1_{A/k})$. Let $X \in \Mat_{n \times n}(\O^\bullet_{A/k})$. Noting that, for any column vector $v \in
({\O^m_{A/k}})^{\oplus n}$, one has $d(Xv) = (dX)(v) + (-1)^m X(dv)$, it follows that 
$$
[\n, X] = dX + [\theta, X].
$$
Thus, since $\tr([\theta, X]) = 0$ (Remark~\ref{tr}), we get $\tr([\n, X]) = \tr(dX) = d \tr(X).$
\end{proof}

\subsection{Statement of the main theorem}
\label{statement}
The following theorem is the main result of this paper:

\begin{thm} \label{mainthm} Let $k$ be a field of characteristic $0$. 
Assume $\cA = (A,h)$ is a $\G$-graded, essentially smooth curved $k$-algebra, and let $\cP = (P, \d_P) \in \Perf(\cA^\op)$ be a
  perfect right $\cA$-module. 
For any connection $\n$ on $\cP$, we have
$$
ch^{II}_{HKR}(\cP) = \tr(\exp(-R)) \in H_0(\Omega^\bullet_{A/k}[[u]], u d + dh),
$$
where $R = \n^2 u + [\n, \d_P]$, and 
$$
\exp(-R) = \id - R + \frac{R^2}{2!} -  \frac{R^3}{3!} + \cdots \in \End_{A}(P) \otimes_A \Omega^\bullet_{A/k}[[u]].
$$
\end{thm}

\begin{rem}
As a reality check, let us apply Lemma~\ref{lem326b} to show that the element $\tr(\exp(-R)) \in (\Omega^\bullet_{A/k}[[u]], u d + dh)$ in the statement of
Theorem~\ref{mainthm} is indeed a cycle. We must show 
$$
ud (\tr(\exp(-R))) = -dh \wedge \tr(\exp(-R)).
$$
Applying Lemma~\ref{lem326b}, we have
$$
ud (\tr(\exp(-R))) = u\tr \circ [\n, \exp(-R)].
$$
Since, $\d_P$ is $A$-linear (Remark~\ref{linear}), line (2) in Remark~\ref{tr} implies
$$
u\tr \circ [\n, \exp(-R)] = \tr \circ [u\n + \d_P, \exp(-R)].
$$
Set $D:= u\n + \d_P$. It's easy to check that $[D, R] = dh$ (here, $dh$ denotes left multiplication by the form $dh$), and an easy induction argument shows
$$
[D, (-R)^i] = -i (-R)^{i-1} [D, R] = -i(-R)^{i-1}dh.
$$
Thus, $\tr \circ [D, \exp(-R)] = -dh \wedge \tr(\exp(-R)).$ 

In a similar way, one can prove directly that the class represented by $\tr(\exp(-R))$ in $H_0(\Omega^\bullet_{A/k}[[u]], u d + dh)$
is independent of the choice of connection $\n$ (of course, this is also a consequence of Theorem \ref{mainthm}).
\end{rem}

\begin{rem} \label{rem72}
For any integer $m \geq 0$ we have
$$
R^m = \sum_{p=0}^m \sum_{j_0+ \dots + j_p  = m-p} \n^{2j_0} \d_P'   \n^{2j_1} \d_P'   \cdots \d'_P \n^{2j_p} u^{m-p},
$$
where $\d'_P := [\n, \d]$, and the inner sum ranges over all $p+1$ tuples of non-negative integers that sum to $m-p$. 
It follows that
\begin{equation} \label{E72b}
\tr(\exp(-R)) = \sum_{n,J \ge 0}^\infty (-1)^{J+n} \sum_{j_0 + \cdots + j_n = J} \frac{\tr \left(\n^{2j_0} \d_P'   \n^{2j_1} \d_P'   \cdots \d'_P \n^{2j_n}
  \right)}{(J+n)!} u^J.
\end{equation}
\end{rem}

\begin{rem}
\label{levi}
Every $(P, \d_P) \in \qPerf(\cA)$ may be equipped with a connection, and so Theorem~\ref{mainthm}
does, in fact, give a
formula for $ch^{II}_{HKR}$ in general. Indeed, assume first that $F$ is a free $A$-module of finite rank. Equip $F$ with a basis $\{e_1,
\dots, e_r\}$, and define 
$\n_F: F \to  F  \otimes_A \Omega^1_{A/k}$
by $\sum a_ie_i \mapsto \sum (-1)^{|e_i|}  e_i \otimes da_i$. 
When $P$ is an arbitrary finitely generated projective $A$-module, we may choose a finite rank free $A$-module $F$ and maps
$$
i : P \to F \text{, } \pi: F \to P
$$
such that $\pi \circ  i = \id_P$; then $\n_P:=(\id \otimes \pi) \circ \n_F \circ i$ is a connection on $P$. 
A connection constructed in this manner is sometimes called a \emph{Levi-Civita connection} (cf. \cite[8.1.1]{Loday}).
\end{rem}

\begin{cor} 
\label{maincor}
Given $\cP = (P, \d_P) \in \Perf(\cA^\op)$, if $P$ admits a flat connection $\n$ (that is, a connection $\n$ such that $\n^2 = 0$), then
$$
ch^{II}_{HKR}(\cP)=\tr(\exp(- \d'_P) )= \sum_{j \ge 0}  \frac{\tr\left([\n, \d_P]^{2j}\right)}{(2j)!}  \in H_0(\Omega^\bullet_{A/k}[[u]], u d + dh).
$$
\end{cor}

\begin{proof}
This follows from Theorem~\ref{mainthm}, using that $\tr(\alpha \otimes \omega) = 0$ for $\a \in \End_A(P)$ and $\o \in \Omega^{\bullet}_{A/k}$ whenever $\a$ 
has odd degree (see Remark~\ref{tr}).
\end{proof}

\subsection{Key technical result}
Let $\cD$ denote a
full cdg subcategory of $\qPerf(\cA^{op})$ consisting of
objects with trivial differential. Let $\cP$ be an object in $\cD$ and $\n$ a connection on $\cP$. The key to proving Theorem \ref{mainthm}
is the construction of a map
$$
\oHN^{II}(\cD) \to (\Omega^\bu_{A/k}[[u]], ud + dh)
$$
of dg-$k[[u]]$-modules which sends $\id_P \in \oHN^{II}(\End(\cP)) \subseteq \oHN^{II}(\cD)$ to $\tr(\exp(-R))$.
The existence of such a map is the content of Theorem \ref{existence}.

\begin{defn} 
\label{tracedefinition}
Let $\cD$ be a full subcategory of $\qPerf(\cA^\op)$ consisting of objects with trivial differentials. 
Let $\n$ denote a family of connections 
$$\n_P: P \to P \otimes_A \Omega^1_{A/k}$$
indexed by the objects $(P, 0) \in \cD$. Define
$$
\tr_\n: \oHoch^{II}(\cD)[[u]] \to \O^\bullet_{A/k}[[u]] 
$$
to be the $k[[u]]$-linear map given as follows: for $n \ge 0$, objects $P_0, \dots, P_n$ of $\cD$, and morphisms
$$
P_0 \xla{\a_0} P_1 \xla{\a_1} P_2 \xla{\a_2}  \cdots \xla{\a_{n-1}} P_n \xla{\a_n} P_0,
$$
define
\begin{equation} \label{trdef}
\tr_\n\left(\a_0 [\ov{\a_1} | \cdots | \ov{\a_n}]\right)
=  \sum_{J=0}^\infty \sum_{j_0 + \cdots+  j_n = J} (-1)^J\frac{\tr\left(\a_0 \n_1^{2j_0} \a_1' \n_2^{2j_1} \a_2' \cdots \n^{2j_{n-1}}_n \a_n' \n^{2j_n}_0\right) }{(J +
  n)!} u^J,  \\
\end{equation}
where $\n_i := \n_{P_i}$ (with $\n_{n+1} = \n_0$),
$\a_i' := \n_i \circ \a_i - (-1)^{|\a_i|} \a_i \circ \n_{i+1}$,
and the inner sum ranges over $(n+1)$-tuples of non-negative integers that sum to $J$.
\end{defn}

\begin{rem}
Note that the coefficient of $u^J$ in formula (\ref{trdef}) is 0 for $J \gg 0$, since $\Omega^i_{A/k} = 0$ for $i > \dim(A)$ (cf. Remark \ref{finitesum}).
\end{rem}

\begin{ex} If $\cD$ consists of just one object $(P,0)$, then $\n$ consists of a choice of connection for $P$, and we have
$$
\tr_\n\left(\a_0 [\ov{\a_1} | \cdots | \ov{\a_n}]\right)
=  \sum_{J=0}^\infty \sum_{j_0 + \cdots+  j_n = J} (-1)^J\frac{\tr\left(\a_0 \n^{2j_0} \a_1' \n^{2j_1} \a_2' \cdots \n^{2j_{n-1}} \a_n' \n^{2j_n}\right) }{(J +
  n)!} u^J,  \\
$$
where $\a_i \in \End_A(P)$ for all $i$, and $\a'_i$ is the derivative of $\a_i$ with respect to $\n$. 
\end{ex}

\begin{ex} 
\label{flatconnections}
If each $\n_i$ is a flat connection, then
$$
\tr_\n(\a_0 [\ov{\a_1} | \cdots | \ov{\a_n}])  = 
\frac{\tr\left(\a_0 \a_1'  \cdots \a_n' \right) }{n!}.
$$
\end{ex}

\begin{ex} \label{trHKR} 
Suppose $\cD$ consists of just the object
$\cP = (A,0) \in \qPerf(\cA^\op)$. Let $\theta : A \xra{\cong} \End_A(A)$ denote the isomorphism of $k$-algebras given by $\theta(a)(x) = ax$, and let 
$\n$ be the de Rham differential $A \xra{d} \O^1_{A/k}$. 
For any $a \in A$, we have
$$
[\n, \theta(a)](x) = d(ax) - (-1)^{|a|} a dx = da \cdot x + (-1)^{|a|} a dx - (-1)^{|a|} a dx = da \cdot x
$$
for all $x \in A$, and thus $[\n, \theta(a)]$ coincides with left multiplication by $da$ in $\End_A(A) \otimes \Omega^\bullet_{A/k}$. 
Since $\n$ is flat, Example \ref{flatconnections} implies
$$
\tr_\n: \oHoch^{II}(A)[[u]] \xra{\theta_*} \oHoch^{II}(\End_A(A))[[u]] \xra{\tr_\n} \O^\bullet_{A/k}[[u]]
$$
is given by
$$
\tr_\n \left(a_0[a_1| \cdots |a_n]\right) = 
\frac{a_0 da_1 \cdots da_n}{n!}.
$$
Thus, $\tr_\n$ coincides, in this case, with the Hochschild-Kostant-Rosenberg isomorphism $\e$ (see Theorem~\ref{HKR}).
\end{ex}

The key result concerning the map $\tr_\n$ is the following:

\begin{thm} 
\label{existence} Let $k$ be a field of characteristic $0$,
let $\cA = (A,h)$ be an essentially smooth curved $k$-algebra, and let $\cD$ be a full subcategory of 
$\qPerf(\cA^\op)$ consisting of objects with trivial differentials. 
For any choices of connections $\n$ on the objects of $\cD$, the map $\tr_\n$ induces a morphism
$$
\oHN^{II}(\cD) \to (\O^\bullet_{A/k}[[u]], ud + dh)
$$
of dg-$k[[u]]$-modules; that is,
$$
\tr_\n \circ (b_2 + b_0 + uB) = (ud + dh) \circ \tr_\n.
$$
\end{thm}

\begin{rem} \label{rem319}
  It is clear from the definition of $\tr_\n$ that it is natural with respect to inclusions of full subcategories (provided the same connections are used on
  the smaller category).
\end{rem}

\begin{rem} \label{rem320}
The proof shows that
$$
\tr_\n \circ (b_2 + uB) = ud \circ \tr_\n
\and
\tr_\n \circ b_0 = dh \circ \tr_\n
$$
both hold.
\end{rem}

We relegate the highly technical proof of Theorem~\ref{existence} to the appendix;
see Corollaries \ref{cor321} and \ref{cor321b}.

\subsection{Proof of the main theorem}

Recall the map
$$
ch_{\cP}^{II}: \HN_0^{II}(\End(\cP)) \to \HN_0^{II}(\cA)
$$
given in Definition \ref{def41}. When $\cA$ is an essentially smooth curved $k$-algebra, we compose
$ch_{\cP}^{II}$ with the HKR isomorphism to obtain the map
$$
ch_{\cP, HKR}^{II}: \HN_0^{II}(\End(\cP)) \to (\O_{A/k}^\bullet [[u]], ud + dh).
$$

\begin{cor}
\label{cor72}
Let $\cA$ be an essentially smooth curved $k$-algebra, $\cP = (P, 0) \in \qPerf(\cA^\op)$ a perfect right $\cA$-module with trivial differential, and $\n$ a connection on $\cP$. 
Then the map 
$$
\HN_0^{II}(\End(\cP)) \to H_0(\Omega_{A/k}^\bullet [[u]], ud + dh)
$$
induced by $\tr_{\n}$ is equal to $ch^{II}_{\cP, HKR}$.
\end{cor}

\begin{proof} Let $\cD$ be the full subcategory of $\qPerf(\cA^\op)$ consisting of just the two objects $(P,0)$ and $(A,0)$. Let $\n_P = \n$ (the connection in
  the statement) and $\n_A = d$, the de Rham differential.
By Theorem~\ref{existence}, the map
$$
\tr_\n: \HN_0^{II}(\cD) \to H_0(\Omega_{A/k}^\bullet [[u]], ud + dh)
$$
defined by \eqref{trdef} is a map of dg-$k[[u]]$-modules.
Moreover, using the naturality of $\tr_\n$ (Remark \ref{rem319}) and Example \ref{trHKR}, we see that the diagram
$$
\xymatrix{
\HN_0^{II}(\End(\cP) ) \ar[dr]_-{\tr_{\n_P}} \ar[r]^-{(\iota_1)_*} & \HN_0^{II}(\cD) \ar[d]^-{\tr_{\n}} & \HN_0^{II}(\cA) \ar[l]_-{(\iota_2)_*}   \ar[dl]^-{\epsilon} \\
& H_0(\Omega_{A/k}^\bullet [[u]], ud + dh) \\
}
$$
commutes, where $\epsilon$ is the HKR map, and $\iota_1, \iota_2$ are the evident inclusion functors.
These inclusion functors are pseudo-equivalences; thus, by Proposition~\ref{functoriality}, the horizontal maps in the diagram are
isomorphisms. Since the HKR map $\epsilon$ is an isomorphism, so are $\tr_{\n}$ and $\tr_{\n_P}$.
Now, consider the composition 
$$
T: \HN_0^{II}(\cD) \to \HN_0^{II}(\qPerf(\cA^\op)) \xla{\cong} \HN_0^{II}(\cA) \xrightarrow{ \epsilon} H_0(\Omega_{A/k}^\bullet
[[u]], ud + dh),
$$
where the first map is induced by inclusion.
The map $T$ also satisfies $T \circ (\iota_2)_* = \epsilon$, and hence we must have
$T = \tr_{\n}$. It follows that
$$
\tr_{\n_P} = \tr_{\n} \circ (\iota_1)_* = T \circ (\iota_1)_*,
$$
and it is clear from the definitions that
$T \circ (\iota_1)_* = ch^{II}_{\cP, HKR}$. 
\end{proof}

We now prove the main theorem:

\begin{proof}[Proof of Theorem~\ref{mainthm}] 
By Proposition~\ref{comparison}, we have 
$$
ch^{II}_{HKR}(\cP)= ch^{II}_{\cP^\nat, HKR}(\sum_n  (-1)^n [ \underbrace{\d_P | \cdots | \d_P}_{n \text{ copies}}]).
$$
Corollary~\ref{cor72} applied to $\cP^\nat$ thus yields
$$
ch^{II}_{HKR}(\cP) = \sum_n \sum_{j_0, \dots, j_n}   (-1)^{j_0 + \cdots + j_n +n} \frac{\tr(\n^{2j_0} \d'_P  \n^{2j_1} \d'_P \cdots \d'_P \n^{2j_n})}{(j_0 +
  \cdots + j_n +n )!} u^{j_0 + \cdots + j_n},
$$
which, by~(\ref{E72b}), is equal to $\tr(\exp(-R))$.
\end{proof}

\begin{rem}
\label{segalremark}
In the setting of Theorem \ref{mainthm}, it seems likely that one can obtain an explicit formula for $ch^{II}_{HKR}$ using instead the results of
Segal in \cite{segal}, in the following  way. In \cite[Proposition 2.13]{segal}, Segal constructs an explicit quasi-isomorphism 
$$
Tr: \Hoch^{II}(\cD) \xra{\simeq} \Hoch^{II}(\cA),
$$
where $\cD$ is the full cdg subcategory of $\qPerf(\cA^{\op})$ spanned by objects with trivial differential. The map
$Tr$ is an adaptation of the ``generalized trace map''
(cf. \cite[Section 1.2]{Loday}) to the setting of curved modules. Unfortunately, $Tr$ does not induce a map on negative cyclic homology of the second kind, because it does not commute with the $B$ operator. For example, 
take $A = \C[x_1, x_2, x_3]/(x_1^2 + x_2^2 + x_3^2 - 1)$ (concentrated in degree 0), $h = 0$, and $P$ the image of the idempotent 
$$
\frac{1}{2} \begin{pmatrix}
1 - x_1 & -x_2 -ix_3 \\
-x_2 + ix_3 & 1 + x_1\\
\end{pmatrix}: A^{\oplus 2} \to A^{\oplus 2}.
$$
Then $(B \circ Tr)(\id_P) \ne (Tr \circ B)(\id_P)$. 

We believe that this problem may be rectified by working in the more general setting of non-unital cdga's, and slightly
modifying Segal's map so that it is defined using the version of the cyclic bar complex for nonunital cdga's (cf. \cite[Section 3.2]{shkly}). 
Granting this, it follows that the modified version of the map $Tr$ induces a map $HN_0^{II}(\cD) \to HN_0^{II}(\cA)$. 
Moreover, the proof of Corollary \ref{cor72} would then hold when $\tr_\n$ is replaced by $\e \circ Tr$ (except there is no need to choose any connections),
and the above proof of Theorem
\ref{mainthm} would therefore yield another explicit formula for $ch^{II}_{HKR}$.

In more detail, given a perfect left $\cA$-module $(P, \delta_P)$, one may realize it as a summand of $(F, \delta_F)$ with $F$ a free $A$-module. 
Let $e$ be the idempotent endomorphism of $F$ with image $P$. Choose a basis of $F$  and, for any endomorphism $\gamma$ of $F$, write $\gamma'$ for
its derivative with respect to the associated Levi-Cevita connection (that is, if we represent $\gamma$ as a matrix, then $\gamma' = d\gamma$, the result of applying
the de Rham differential to the entries of this matrix).
Then, assuming the extension of Segal's map to the non-unital setting works out as we expect,
the formula for the Chern character of $(P, \delta_P)$ arising from Segal's map would be
\begin{equation} \label{SegalChern}
\begin{aligned}
& \sum_j \frac{1}{j!} tr(e (\d'_F)^j) \\
& + \sum_{i \geq 1} \sum_{j_0, \dots, j_{2i}}
(-1)^i \frac{(2i-1)!}{(i-1)!} \frac{1}{(2i +J )!} \tr\left((2e-1) (\d'_F)^{j_0} e' (\d'_F)^{i_1} e' \cdots e' (\d'_F)^{j_{2i}}\right),
\end{aligned}
\end{equation}
where $j_1,  \dots, j_{2i}$ range over all non-negative integers, and
$J := j_0 + \cdots + j_{2i}$.

Assuming the details check out, the complicated formula \eqref{SegalChern} and the formula of Theorem \ref{mainthm} must agree as homology classes,
since they both coincide with $ch^{II}_{HKR}(P, \d_P)$. It would be pleasing to give a direct proof
of this fact, but we have been unable to do so.
\end{rem}

\section{Examples}
\label{examplesection}
Throughout this section, assume $\on{char}(k) = 0$.

\begin{ex}
\label{mfchern}
Suppose $\G = \Z/2$ and $\cA = (A, -f)$, where $A$ is the localization of a smooth $k$-algebra at a maximal ideal $\fm$, and $f \in \fm \subseteq A$ is a non-zero-divisor. Recall that $\Perf(\cA^\op)$ is identical to the differential $\Z/2$-graded category $\MF(A, f)$ of matrix factorizations of $f$. The HKR isomorphism yields
$$
\HN_0^{II}(\cA) \xra{\cong} H_0(\Omega^\bullet_{A/k}[[u]], ud - df).
$$
Let $\cP= (P, \d_P) \in \Perf(\cA^{\op})$, and write $P = P_0 \oplus P_1$, where $P_0$ (resp. $P_1$) is the even (resp. odd) degree component of $P$. Since $f$ is a
non-zero-divisor, $r:=\rk(P_0) = \rk(P_1)$. Upon choosing bases of $P_0$ and $P_1$, we may identify $\d_P$ with a matrix of the form $\begin{pmatrix} 0 & \alpha \\ \beta & 0 \end{pmatrix}$, where $\alpha:
P_1 \to P_0$ and $\beta: P_0\to P_1$ are $(r \times r)$ matrices over $A$.  Using this choice of basis, we also construct a Levi-Civita connection $\n$ on $\cP$
(Remark~\ref{levi}).  Since this connection is flat,
we have 
$$
R = [\n, \d_P] = \begin{pmatrix} 0 & d\alpha \\ d\beta & 0 \end{pmatrix} \in \End_A(P) \otimes_A \Omega^\bullet_{A/k}[[u]],
$$
where $d\alpha, d\beta$ denote the matrices obtained by applying the de Rham differential entry-wise to $\alpha, \beta$. Thus, for any $j \ge 0$, we have
$$
R^{2j} = [\n, \d_P]^{2j} = \begin{pmatrix} (d\alpha d\beta)^j & 0 \\ 0 & (d\beta d\alpha)^j \\ \end{pmatrix},
$$
and so $\tr(R^{2j}) = \tr((d\alpha d\beta)^j) -  \tr((d\beta d\alpha)^j)$.
By Remark~\ref{tr},
\begin{displaymath}
   \tr([\n, \d_P]^{2j})  = 
\begin{cases}
       0 &  \text{ if $ j = 0$} \\
       2 \tr((d\alpha d\beta)^j) &  \text{ if $j > 0$.} \\
     \end{cases}
\end{displaymath} 
By Corollary~\ref{maincor},
$$
ch^{II}_{HKR}(\cP) = \sum_{j \ge 1}  \frac{ 2\tr((d\alpha d\beta)^j )}{(2j)!}\in H_0(\Omega^\bullet_{A/k}[[u]], ud - df).
$$

Recall that the canonical map $\HN^{II}_0(\cA) \to \HH^{II}_0(\cA)$ is given by setting $u = 0$, and, in this case, this map may be identified, via the HKR isomorphism,
with 
$$
H_0(\Omega^\bullet_{A/k}[[u]], ud - df) \xra{u \mapsto 0} H_0(\Omega^\bullet_{A/k}, -df).
$$
Let $ch^{II}_{HH}(\cP)$ denote the image of $ch^{II}_{HN}(\cP)$ in $\HH^{II}_0(\cA)$. The class $ch^{II}_{HH}(\cP)$ corresponds, via the HKR isomorphism, to the class
\begin{equation}
\label{HHchern}
\sum_{j \ge 1}  \frac{ 2\tr((d\alpha d\beta)^j )}{(2j)!}  \in H_0(\Omega^\bullet_{A/k}, -df).
\end{equation}
As discussed in the introduction, Segal obtains formula \eqref{HHchern} in \cite[3.1]{segal}. 

When $A/f$ has at most an isolated singularity (meaning that $(A_\fp/f)$ is smooth for all primes $\fp \ne \fm$), then, letting $n = \dm(A)$, we have 
$$
H_0(\Omega^\bullet_{A/k}, -df) \cong
\begin{cases}
\frac{\Omega^n_{A/k}}{df \smsh \Omega^{n-1}_{A/k}}, & \text{ $n$ even} \\
0 , & \text{ $n$ odd}. \\
 \end{cases}
$$
Moreover, when $n$ is even, $ch^{II}_{HN}(\cP)$ corresponds to
\begin{equation}
\label{pvformula}
\overline{\frac{2 tr((d\a d\b)^{\frac{n}{2}})}{n!}} \in \frac{\Omega^n_{A/k}}{df \smsh \Omega^{n-1}_{A/k}}.
\end{equation}
In this case, the canonical map $\HH_0(\MF(A,f)) \to \HH^{II}_0(\MF(A, f))$ is an isomorphism \cite[Section 4.8]{PP}. Thus, \eqref{pvformula}
coincides with the formula,
due to Polishchuk-Vaintrob, of the Chern character of a matrix
factorization of an isolated singularity $f \in k[x_1, \dots, x_n]$ taking values in $\HH_0(\MF(A,f))$; see 
\cite[Corollary 3.2.4]{PV} for the precise statement of  their result.
\end{ex}

\begin{ex}
\label{unstable}
Let $A$ be an essentially smooth $k$-algebra, let $f_1, \dots, f_c \in A$ be a regular sequence, and set $R: = A/(f_1, \dots, f_c)$. Take $\G = \Z$. 
Let $D^b_{dg}(R)$ denote the differential $\Z$-graded category whose objects are bounded below complexes of finitely generated projective
$R$-modules whose total homology is finitely generated, so that $D^b_{dg}(R)$ is a dg enhancement of the ordinary derived category $D^b(R)$. In this example, we give a formula for the Chern character map 
$$
K_0(D_{dg}^b(R)) \to \HN_0(D^b_{dg}(R)).
$$

Let $G_0(R)$ denote the Grothendieck group of the exact category $\on{mod}(R)$ of finitely generated $R$-modules. Then there is an isomorphism
\begin{equation} \label{E627b}
G_0(R) \xra{\cong} K_0(D_{dg}^b(R)) = K^\Delta_0(D^b(R))
\end{equation}
given by sending the class of a module $M$ to the class of $M$ regarded as a complex concentrated in degree $0$.
The inverse sends the class of a complex $C$ to $\sum_i (-1)^i [H^i(C)]$.
Note that the Grothendieck group of $D^b_{dg}(R)$ coincides, by definition, with $K^\Delta_0([D^b_{dg}(R)]) = K_0^\Delta(D^b(R))$ \cite[Section 3.2.32]{schlichting}.

Let $\cB$ denote the curved differential $\Z$-graded algebra $(A[T_1, \dots, T_c],  -(f_1T_1 + \cdots + f_cT_c))$, where $T_1, \dots, T_c$ are degree two indeterminants. A theorem
of Burke-Stevenson 
\cite[Theorem 7.5]{burke2013}  gives a quasi-equivalence of dg categories
\begin{equation}
\label{BS}
D_{dg}(R) \xra{\simeq} \Perf(\cB) 
\end{equation}
and therefore a commutative diagram
\begin{equation} \label{E627c}
\xymatrix{
K_0(D_{dg}^b(R)) \ar[r] \ar[d]^\cong & \HN_0(D^b_{dg}(R)) \ar[r] \ar[d]^\cong & \HN^{II}_0(D^b_{dg}(R)) \ar[d] \\
K_0(\Perf(\cB)) \ar[r] & \HN_0(\Perf(\cB)) \ar[r]^{\cong} & \HN_0^{II}(\Perf(\cB)). 
}
\end{equation}
The two left vertical arrows are isomorphisms because $K_0( - )$ and $\HN_0( - )$ send quasi-equivalences to isomorphisms. The bottom-right map is an isomorphism by
\cite[Corollary A in Section 4.7]{PP}.

Combining \ref{E627b}, Diagram \eqref{E627c}, Proposition \ref{thmPP}, and the HKR theorem, one sees that the Chern character map
$$K_0(D_{dg}^b(R)) \to \HN_0(D^b_{dg}(R))$$
is given, up to the canonical isomorphisms described above, by a map

\begin{equation} \label{E627}
ch : G_0(R) \to H_0(\Omega^\bullet_{A[T_1, \dots, T_c]/k}[u], ud - d(f_1T_1 + \cdots + df_cT_c)).
\end{equation}
In detail, $ch$ is given by the composition
$$
G_0(R)   \xra{(\ref{E627b})}  K_0(D^b_{dg}(R))  \xra{(\ref{BS})}   K_0(\Perf(\cB))  \xra{ch^{II}_{HKR}} H_0(\Omega^\bullet_{A[T_1, \dots, T_c]/k}[u], ud - d(f_1T_1 + \cdots + df_cT_c)).
$$
We proceed to describe this map explicitly.

Let $M$ be a finitely generated $R$-module. We first describe the image of $[M]$ under the map
$$
\psi: G_0(R) \xra{\cong} K_0(\Perf(\cB))
$$
induced by (\ref{E627b}) and (\ref{BS}). 

For any $J = (j_1, \dots, j_c) \in \Z^c_{\ge 0}$, set $| J | = \sum_{i = 1}^c j_i$. For each $1 \le i \le c$, let $e_i$ denote the element of $\Z^c_{\ge 0}$ with $1$ in the
$i^{\on{th}}$ component and 0 elsewhere.  
Given a finite $A$-projective resolution $P$ of $M$, we equip  $P$ with a \emph{system of higher homotopies}; i.e., a family of endomorphisms 
$$
\sigma_J : P \to \Sigma^{-2|J| + 1}P
$$
indexed by $J \in \Z^c_{\ge 0}$ such that
\begin{itemize}
\item $\sigma_0$ is the differential on $P$,
\item for each $1 \le i \le c$, the map $\sigma_0 \sigma_{e_i} + \sigma_{e_i} \sigma_0$ is given by multiplication by $f_i$, and
\item for each $J$ with  $|J| \ge 2$, $\sum_{J_1 + J_2 = J} \sigma_{J_1} \sigma_{J_2} = 0$.
\end{itemize}
(See \cite[Proposition 3.4.2]{eisenbudpeeva} for a proof that such a system of higher homotopies exists.) Then 
$$
\psi([M]) = [(P[T_1, \dots, T_c], \d)],
$$ 
where $\d := \sum_{J \in \Z^c_{\geq 0}} \s_J T^J$. (For $J = (i_1, \dots, i_c)$, $T^J$ denotes  $T_1^{i_1} \cdots T_c^{i_c}$.) Thus,
\begin{equation}
\label{chernci}
ch([M]) = (ch_{HKR}^{II} \circ \psi )([M])  = ch_{HKR}^{II}( P[T_1, \dots, T_c], \d  ),
\end{equation}
and one may compute $ch_{HKR}^{II}( P[T_1, \dots, T_c], \d  )$ using Theorem \ref{mainthm}.

We work out this computation in detail in the case where $A$ is local and $c = 1$. Let $M$ be a finitely generated $R$-module.
The group
$G_0(R)$ is generated by classes represented by \emph{maximal Cohen Macaulay} (MCM) $R$-modules, i.e. $R$-modules whose projective dimension over
$A$ is 1, so we may assume $M$ is MCM. Let $P_1 \xra{\alpha} P_0$ be the minimal $A$-free resolution of $M$; $\alpha$ was denoted by $\s_0$
in the more general setup above. 
Then there is a unique map $\b: P_1 \to P_0$ such that $\a \circ \b = f = \b \circ \a$; this is the map written as $\s_{e_1}$ above.
Therefore, by (\ref{chernci}), 
$$
ch([M]) = ch_{HKR}^{II}( P_0[T] \oplus \Sigma P_1[T], \d  ),
$$
where
$$
\delta := 
\begin{pmatrix}
0 & \a \\
\b T & 0 \\
\end{pmatrix}.
$$

We compute $ch_{HKR}^{II}( P_0[T] \oplus \Sigma P_1[T], \d  )$ using Theorem \ref{mainthm}. It is convenient to use
the canonical isomorphism
$$
\Omega^\bullet_{A[T]/k} \cong \Omega^\bullet_{A/k}[T] \oplus \Omega^\bullet_{A/k}[T] dT.
$$
Under this identification, the differential $ud - d(fT) = ud - df T - fdT$ on $\Omega^\bullet_{A[T]/k}$ becomes
$$
\begin{pmatrix}
ud_A - (df)T & 0 \\
\frac{\partial}{\partial T}dT-fdT & ud_A - (df)T \\
\end{pmatrix},
$$
where $d_A$ is the de Rham differential on $\Omega^\bullet_{A/k}$.

Choose bases of $P_0$ and $P_1$, so that we may represent $\a, \b$ as matrices. Let $\n$ denote the Levi-Civita connection on $P_0[T] \oplus \Sigma P_1[T]$ corresponding to this choice of basis. Then the associated curvature, regarded as a matrix with entries in $\Omega^\bullet_{A/k}[T] \oplus \Omega^\bullet_{A/k}[T] dT$, is
$$
R = [\n, \d] = 
\begin{pmatrix}
0 & d\a \\
T d\b & 0 \\
\end{pmatrix} +
\begin{pmatrix}
0 & 0 \\
\b   & 0 \\
\end{pmatrix} dT.
$$

For brevity, we set 
$$
R_1 = \begin{pmatrix}
0 & d\a \\
T d\b & 0 \\
\end{pmatrix} \and
R_2 = \begin{pmatrix}
0 & 0 \\
\b   & 0 \\
\end{pmatrix},
$$
so that 
$R = R_1 + R_2 dT$. For any $l$, we have
$$
R^l = R_1^l + \sum_{i=0}^{l-1} R_1^i R_2 dT R_1^{l-i-1}.
$$
By properties of the trace map,
$$
\tr(R^l) = \tr(R_1^l) + l \cdot \tr(R_1^{l-1} R_2) dT,
$$
which gives
$$
\tr(e^{-R}) = \tr(e^{-R_1}) - \tr(e^{-R_1} R_2) dT.
$$
Therefore, 
$$
ch([M]) = \sum_{j \ge 1}  \frac{ 2\tr((d\alpha \cdot d\beta)^j )}{(2j)!} T^j
+ 
\sum_{j \ge 0} \frac{ \tr \left( (d\a \cdot d\b)^j \cdot d\a \cdot \b\right) }{(2j+1)!} T^j dT,
$$
regarded as a cochain in the complex
$$
\left(\Omega^\bullet_{A/k}[T,u] \oplus \Omega^\bullet_{A/k}[T,u] dT, 
\begin{pmatrix} u d_A  - T df & 0 \\
\frac{\partial}{\partial T} dT - f dT &  u d_A - T df
\end{pmatrix}
\right).
$$
Observe that the computation in the previous example is recovered from this one by composing with the map to the $\Z/2$-graded complex
$(\Omega^\bullet_{A/k}[[u]], u d_A + df)$ given by setting $T = 1$ and $dT = 0$. 
\end{ex}

\begin{ex}
\label{nonflat}
As a final example, we compute the Chern character of a perfect curved module $\cP$ which does not admit a flat connection. Take $\G = \Z/2$, $k = \C$, and $\cA = (A, -h)$, where $A$ is the
$\G$-graded smooth $\C$-algebra $\C[x_1, \dots, x_5]/(x_1^2 + \cdots + x_5^2 - 1)$ concentrated in even degree, and $h = (1 - x_1^2)(x_2x_3 + x_4x_5)$. Recall that $\Perf(\cA^{\op})$
coincides with the dg category $\MF(A, h)$ of matrix factorizations of $h$. Let $P$ denote the image of the idempotent 

$$e: = \frac{1}{2}\begin{pmatrix} 
1-x_1 & -x_2 -ix_3 & x_4 + ix_5 & 0 \\
-x_2 +ix_3 & 1 + x_1 & 0 & x_4 + ix_5 \\
x_4 -ix_5 & 0 & 1 + x_1 & x_2 + ix_3 \\
0 & x_4 - ix_5 & x_2 - ix_3 & 1 - x_1 \\
 \end{pmatrix}$$
 of $A^{\oplus 4}$. A  direct
 calculation shows that the top Chern class of the ordinary projective $A$-module $P$, thought of as an element of $H_{dR}^4(A/k)$, is 
$$\frac{3}{2} \sum_{i = 1}^5 (-1)^{i+1} x_i dx_1 \cdots \widehat{dx_i} \cdots dx_5.$$
Stokes' Theorem implies that the integral of this Chern class, thought of as an element of $H_{dR}^4(S^4)$, is $\frac{15}{2} \on{vol}(B) \ne 0$, where $B$ denotes the unit ball in
$\mathbb{R}^5$. In particular, $P$ does not admit a flat connection.

Interpret the idempotent $e$ as a map $\pi: A^{\oplus 4} \to P$, and let $\iota : P \into A^{\oplus 4}$ denote the inclusion. Similarly, let $P'$ denote the image of $1 - e$, interpret $1 - e$ as a map $\pi' : A^{\oplus 4}  \to P'$, and let $\iota' : P' \into A^{\oplus 4}$ denote the inclusion. We now construct a matrix factorization of $h$ with underlying projective $A$-module $P \oplus \Sigma P'$.

Set $v:= \frac{1}{ 1- x_1}$, and consider the matrices
$$
S : = \begin{pmatrix}
1 & 0 & v(x_2 + ix_3) & v(-x_4 - ix_5)  \\
v(-x_2 + ix_3) & v(x_4 + ix_5)  & 1 & 0  \\
v(x_4 - ix_5) & v(x_2 + ix_3)  & 0 & 1 \\
0 & 1 & v(-x_4 +ix_5) & v(-x_2 + ix_3) \\
\end{pmatrix}
$$
$$
U:=\frac{1}{2}\begin{pmatrix}
1-x_1 & -x_2 - ix_3 & x_4 + ix_5 & 0\\
0 & x_4 - ix_5 & x_2 - ix_3 & 1 - x_1 \\
x_2 - ix_3 & 1-x_1 & 0 & -x_4 - ix_5 \\
-x_4 + ix_5 & 0 & 1-x_1 & -x_2 - ix_3 \\
\end{pmatrix}.
$$
Columns 1 and 2 (resp. 3 and 4) of $S$ give a basis for $P \otimes A[v]$ (resp. $P' \otimes A[v]$) over $A[v]$, and $U = S^{-1}$.

Set $T : = (1-x_1)S$. Then $T$ and $U$ are matrices over $A$, and of course we have
$$TU = (1-x_1)\id_{A^{\oplus 4}} = UT.$$
Let $C$ denote the matrix
$$
\begin{pmatrix}
0 & 0 & 1 & 0 \\
0 & 0 & 0 & 1 \\
1 & 0 & 0 & 0 \\
0 & 1 & 0 & 0 \\
\end{pmatrix},
$$
and set
$$
\a := \begin{pmatrix}
x_2 & -x_4 & 0 & 0 \\
x_5 & x_3 & 0 & 0 \\
0 & 0 & x_2 & -x_4 \\
0 & 0 & x_5 & x_3 \\
\end{pmatrix},
\text{ } \b := \begin{pmatrix}
x_3 & x_4 & 0 & 0 \\
-x_5 & x_2 & 0 & 0 \\
0 & 0 & x_3 & x_4 \\
0 & 0 & -x_5 & x_2 \\
\end{pmatrix}.
$$
Letting $\d:=\begin{pmatrix} 0 &  \pi T \b CU\iota'  \\ \pi' TC\a U\iota & 0 \end{pmatrix}$, one sees that $\cP := (P \oplus P'[1], \d)$ is a matrix factorization of $h = (1-x_1)^2(x_2x_3 + x_4x_5)$. 

We observe that there is nothing special about $x_2x_3 + x_4x_5$ in this construction, and the only thing special about the matrices $\a$ and $\b$ are that the top left and bottom right corners come from a $2 \times 2$ matrix factorization of $x_2x_3 + x_4x_5$. So we could have taken any $h' \in A$ and any $2 \times 2$ matrix factorization $(\a', \b')$ of $h'$ and obtained, in the same way, a matrix factorization of $(1- x_1)^2h'$ with components given by $P$ and $P'$.

Now, let's use Theorem \ref{mainthm} to compute $ch^{II}_{HKR}(\cP)$. Equip $A^{\oplus 4}$ with a Levi-Civita connection $D$ corresponding to the standard basis, and let
$$\n:= \begin{pmatrix} 
e D \iota & 0  \\
0 & e' D \iota' \\
\end{pmatrix}
$$
be the induced connection on $P \oplus P'[1]$. Theorem \ref{mainthm} tells us
$$
ch_{HKR}^{II}(\cP) = \tr(\exp(-R)) = \sum_{n, J \ge 0}^{\infty} (-1)^{J + n} \sum_{j_0 + \cdots + j_n = J}\frac{\tr(\n^{2j_0}\delta' \n^{2j_1}\delta'  \cdots \delta'  \n^{2j_n})}{(J + n)!}u^J,
$$
where $\d' := [\n, \d]$. Collecting the terms which are nonzero in $\Omega_{A/k}^{\bullet}[[u]]$, we get
$$\tr(\exp(-R)) = \frac{\tr((\d')^4)}{24} - \frac{\tr(\d' \n^2 \d')}{2}u +\frac{\tr(\n^4)}{2}u^2.$$
The first term is the image of $\ch_{HH}^{II}(\cP)$ under the HKR map, and the last term is the difference of the top Chern classes of the components $P$ and $P'$.
\end{ex}

\appendix
\label{appendixsection}

\section{Hochschild homology and cohomology via twisting cochains}
\label{appA}

Throughout, $k$ is an arbitrary field, 
$\Gamma = \Z$ or $\Z/2$, and ``graded'' means $\Gamma$-graded. Let $\cA = (A, d_A, h)$ be a cdga over $k$.
In this section, we use the notion of a ``twisting cochain" to construct a dga $(\Hom_k(T,A), d_\tau)$ whose cohomology is the Hochschild cohomology of $\cA$ 
(cf. \cite[Section 2.4]{PP}), 
and we exhibit the complex $\tHoch(\cA)$ defined in Section~\ref{MCunital} as a left dg-module over $(\Hom_k(T,A), d_\tau)$.
Our approach uses ideas of Negron in \cite{negron}. 

Define a graded $k$-vector space
$$
T = T(A)  = \bigoplus_{n \geq 0} (\S A)^{\otimes n},
$$
where $\S A$ is the shift of the graded vector space $A$, so $(\S A )^i = A^{i+1}$. For $a_1, \dots, a_n \in A$, let $[a_1| \cdots | a_n]:= sa_1 \otimes \cdots \otimes sa_n$, where $s: A \to \S A$ is the canonical map of degree $-1$.
Thus, a typical element of $T$ is a $k$-linear combination of
expressions of the form $[a_1| \cdots | a_n]$. Observe that $[a_1| \cdots | a_n] \in T$ has degree $|a_1| + \cdots + |a_n| - n$. As a matter of convention, for
$n=0$, the symbol $[]$ denotes the element $1$ in $k = (\S A)^{\otimes 0}$.

$T$ is a counital  coalgebra over $k$ under the ``splitting of tensors'' coproduct map 
$$
\Delta: T \to T \otimes_k T
$$
defined by
$$
\begin{aligned}
\Delta([a_1| \dots |a_n]) 
& = \sum_{i=0}^n [a_1 | \cdots | a_i] \otimes [a_{i+1} | \cdots | a_n] \\
& = [] \otimes [a_1| \dots |a_n] +
(\sum_{i=1}^{n-1} [a_1 | \cdots | a_i] \otimes [a_{i+1} | \cdots | a_n]) +
 [a_1| \dots |a_n] \otimes []. \\
\end{aligned}
$$

Define a {\em coderivation} on $T$ to be a degree one map $\d: T \to T$ such that 
$$
(1 \otimes \d + \d \otimes 1) \circ \Delta = \Delta \circ \d.
$$
Let $\on{CoDer(T)}$ denote the set of (degree one) coderivations on $T$. Then there is a bijection 
$$
\on{CoDer(T)} \xra{\cong} \Hom_k(T, \S A)^1
$$
given by $\d \mapsto \pi \circ \d$, where $\pi: T \onto \S A$ is the canonical projection.
The inverse is given by sending $g \in \Hom_k(T, \S A)^1$ to the coderivation
$$
[a_1 | \cdots | a_n] \mapsto \sum_{\substack{0 \le j \le n-s \\ s \ge 0}} (-1)^{(|a_1| + \cdots + |a_j| - j)}
[a_1 | \cdots | a_j | g_s([a_{j+1}| \cdots | a_{j + s}]) | \cdots | a_n],
$$
where $g_s$ is the restriction of $g$ to $(\S A)^{\otimes s}$; see \cite[page 41]{markl2012deformation}.

We define coderivations $\d^{(2)}_T$,  $\d^{(1)}_T$, and $\d^{(0)}_T$ on $T$ using the multiplication map $\mu$, differential $d_A$, and curvature $h$ of the cdga $\cA$.
In detail, for $i=0,1,2$, let $\d^{(i)}_T$ be the unique coderivation such that $\pi \circ \d^{(i)}_T$ vanishes on $(\S A)^{\otimes n}$ for all $n \ne i$ and such that
the restriction of $\pi \circ \d^{(i)}_T$ to $(\S A)^{\otimes i}$ is
\begin{itemize}
\item the map $[a | b] \mapsto  (-1)^{|a|-1} [ab]$ for $i = 2$;
\item the map $[a] \mapsto [-da]$ for $i = 1$; and
\item the map $(\S A)^{\otimes 0} \to \S A$ sending $[] \in (\S A)^{\otimes 0} $ to $[h]$ for $ i = 0$.
\end{itemize}

The sign for $\delta_T^{(2)}$ is explained by the fact that it is the  map $s \circ \mu \circ (s^{-1} \otimes s^{-1})$. Applying this to $[a|b] = sa \otimes sb$ introduces a sign of $(-1)^{|s^{-1}||sa|} = (-1)^{|a| -1}$, because $s^{-1}$ and $sa$ are interchanged.
The sign in the formula for $\d_T^{(1)}$ is justified since $[a] \mapsto -[da]$ is the standard differential on the suspension $\S A$ of $A$.

Explicitly, for any $n$ and $a_1, \dots, a_n \in A$, we have:
$$
\d^{(2)} _T([a_1 | \cdots | a_n])  =
\sum_{j=1}^{n-1}  (-1)^{|a_1| + \cdots + |a_j| - j}  [a_1 | \cdots |a_j a_{j+1}| \cdots | a_n],
$$
$$
\d^{(1)}_T([a_1 | \cdots | a_n]) 
 = \sum_{j=1}^n   (-1)^{|a_1| + \cdots + |a_{j-1}| -j}[a_1 | \cdots | d(a_j) | \cdots | a_n],
$$
and
$$
\d^{(0)}_T([a_1 | \cdots | a_n]) 
 = \sum_{j=0}^n (-1)^{|a_1| + \cdots +|a_j| -j}[a_1 | \cdots | a_{j} | h | a_{j+1} | \cdots | a_n].
$$

We endow $T$ with the coderivation $\d_T = \d_T^{(2)} + \d_T^{(1)} + \d_T^{(0)}$.

\begin{lem} $(T, \Delta, \d_T)$ is a dg-coalgebra; that is, $\d_T$ is a coderivation, and $\d_T^2 = 0$.
\end{lem}

\begin{proof} $\d_T$ is a coderivation since each of $\d_T^{(2)}$, $\d_T^{(1)}$ and $\d_T^{(0)}$ are coderivations. The condition $\d_T  \circ \d_T = 0$ is equivalent to the defining relations for a curved dga. In detail, 
we have
$$
\d_T \circ \d_T  = \d^{(2)}_T \circ \d^{(2)}_T 
 + [\d^{(2)}_T , \d^{(1)}_T] 
 + [\d^{(2)}_T, \d^{(0)}_T]  +\d^{(1)}_T \circ \d^{(1)}_T  
 + [\d^{(1)}_T , \d^{(0)}_T] 
 + \d^{(0)}_T \circ \d^{(0)}_T, 
$$
and each of the six terms on the right is a coderivation of $T$. In general, a coderivation $\d'$ of
$T$ vanishes if and only if the $k$-linear map $\pi \circ \d' = 0$. 
We have
\begin{enumerate}
\item $\pi \circ \d^{(2)}_T \circ \d^{(2)}_T = 0$, since multiplication is associative;
\item $\pi \circ [\d^{(2)}_T ,\circ \d^{(1)}_T]   = 0$, since $d$ satisfies the Leibniz rule;
\item $\pi \circ \left( [\d^{(2)}_T ,\d^{(0)}_T] + \d^{(1)}_T \circ \d^{(1)}_T  \right) = 0$, since $d^2 = [h,-]$; 
\item $\pi \circ [\d^{(1)}_T ,\d^{(0)}_T]  = 0$, since $d(h) = 0$; and
\item $\pi \circ \d^{(0)}_T \circ \d^{(0)}_T  = 0$, since $\d^{(0)}_T$ maps $(\S A)^{\otimes n}$ to
  $(\S A)^{\otimes {n+1}}$ for all $n$. 
\end{enumerate}
\end{proof}

We now consider the graded $k$-vector space $\Hom_k(T,A)$. There is a canonical isomorphism 
$$
\Hom_k(T, A) \cong \prod_{n \geq 0} \Hom_k((\S A)^{\otimes n}, A),
$$
so we will write a typical element of $\Hom_k(T, A)$ 
as an infinite sum $f = \sum_n f_n$ for maps $f_n: (\S A)^{\otimes n} \to A$, $n \geq 0$. 
The maps $\d_T$ and $d_A$ endow $\Hom_k(T, A)$ with a differential given by
$$
d_{\Hom}(f) = d_A \circ f - (-1)^{|f|} f \circ \d_{T},
$$
making it into a chain complex.

Using that  $T$ is a coalgebra and $A$ is an algebra, $\Hom_k(T,A)$ 
admits a ``convolution'' product $\star$ defined as 
$$
f \star g = \mu \circ (f \otimes g) \circ \Delta.
$$
Explicitly, given 
$f_m: (\S A)^{\otimes m} \to A$ and
$g_n: (\S A)^{\otimes n} \to A$,  the map $f_m \star g_n: (\S A)^{\otimes (m + n)} \to A$  is given by
$$
(f_m \star g_n)([a_1 | \cdots |a_{m+n}]) = (-1)^{|g|(|a_1| + \cdots + |a_m| - m)} f([a_1| \cdots |a_m]) g([a_{m+1} | \cdots | a_{m+n}]).
$$
For $f = \sum_m f_m$ and $g = \sum_n g_n$ we have
$$
f \star g = \sum_{k} \sum_{m+n = k} f_m \star g_n.
$$
The following is an immediate consequence of the facts that $\Delta$ is co-associative, $\d_T$ is a coderivation for $\Delta$, $\mu$ is associative, and $d_A$ is
a derivation for $A$:

\begin{lem} 
\label{derivation}
The endomorphism $d_{\Hom}$ on  $\Hom_k(T, A)$ is a derivation relative to the convolution product.
\end{lem}

In general, $(\Hom_k(T,A), d_{\Hom})$ is not a dga, because $d_{\Hom}$ does not square to $0$. Instead,
$$
d_{\Hom} \circ d_{\Hom} = [h \circ \e, -],
$$
where $\e: T \to k$ is the counit for $T$ and we interpret $h$ as a map $k \to A$. Since $d_{\Hom}(h \circ \e) = 0$, we have:

\begin{lem} $(\Hom_k(T,A), \star, d_{\Hom}, h \circ \e)$ is a cdga over $k$.
\end{lem}

Let $\tau \in \Hom_k(T,A)$ be the degree $1$ element that is zero on $(\S A)^{\otimes n}$ for all $n \ne 1$ and is the map
$s^{-1}$ on $\S A$. 
Then $\tau$ is a  ``twisting cochain'' for the cdga $(\Hom_k(T,A), \star, d_{\Hom}, h \circ \e)$; that is, it satisfies the Maurer-Cartan equation:
$$
\tau \star \tau - d_{\Hom}(\tau) + h \circ \e = 0.
$$
As with any twisting cochain, we may use $\tau$ to deform the differential by setting
$$
d_\tau = d_{\Hom} - [\tau, -],
$$ 
where $[-, -]$ denotes the commutator with respect to the convolution product. Explicitly, 
$$
d_\tau(f) =d_A \circ f - (-1)^{|f|} f \circ \d_T - \tau \star f + (-1)^{|f|} f \star \tau.
$$

\begin{lem} \label{lem326}
$(\Hom_k(T,A), \star, d_\tau)$ is a dga over $k$. That is, $d_\tau \circ d_\tau = 0$, and $d_\tau$ satisfies the Leibniz rule relative to the
  convolution product $\star$
\end{lem}

\begin{proof} For any degree one element $\b$, the endomorphism $[\b, -]$ satisfies the Leibniz rule; hence, so does $d_\tau$. Using that $\tau$ is a twisting cochain, we get
$$
\begin{aligned}
d_\tau^2  & = d_{\Hom}^2 - d_{\Hom} \circ [\tau,-] -  [\tau,-] \circ d_{\Hom} + [\tau \star \tau,-] \\
& = [h \circ \e,-]  - [d_{\Hom}(\tau), -] + [\tau \star \tau, -] \\
& = [h \circ \e - d_{\Hom}(\tau) + \tau \star \tau, -] \\
& = 0. \qedhere \\ 
\end{aligned}
$$
\end{proof}

Let us write $\Hom_k^\tau(T,A)$ for the dga $(\Hom_k(T,A), \star, d_\tau)$.
Explicitly,  for $a_1, \dots, a_{n+1} \in A$ and $f = \sum_n f_n \in \Hom_k(T,A)$, we have
$$
\begin{aligned}
d_\tau(f)([a_1 | \cdots | a_{n+1}])
= & - (-1)^{|f|(|a_1|-1)} a_1 f_{n}([a_2| \cdots | a_{n+1}]) \\
& - \sum_{j=1}^{n} (-1)^{|f| + |a_1| + \cdots +|a_j|-j} f_{n}([a_1| \cdots |a_ja_{j+1}| \cdots |a_{n+1}]) \\
& + (-1)^{|f| + |a_1| + \cdots + |a_n| -n} f_n([a_1| \cdots | a_{n}]) a_{n+1} \\
&  - \sum_{j=1}^{n+1}   (-1)^{|f|+|a_1| + \cdots + |a_{j-1}| -j}f_{n+1}([a_1 | \cdots | d_A(a_j) | \cdots | a_{n+1}]) \\
& - \sum_{j=0}^{n+1} (-1)^{|f|+|a_1| + \cdots +|a_j| -j} f_{n+2}([a_1 | \cdots | a_{j} | h | a_{j+1} | \cdots | a_{n+1}]) \\
& + d_A(f_{n+1}([a_1| \cdots |a_{n+1}])).
\end{aligned}
$$
Thus, $\Hom^\tau_k(T,A)$ is the standard Hochschild cohomology complex of the cdga $\cA$ (cf. \cite[Section 2.4]{PP}).

We can also build the Hochschild homology complex of $\cA$ from $T$, $\d_T$ and $\tau$. 
We define both a left and a right action of the algebra $\Hom_k(T,A)$ on 
the graded $k$-vector space $A \otimes_k T$ as follows. As a matter of shorthand, given an element $t \in T$,
we write $t_i', t''_i$ for the various elements of $T$ appearing in $\Delta(t) = \sum_i t'_i \otimes t''_i$. In detail,
if $t = [a_1| \cdots |a_n]$, then $t'_i = [a_1| \cdots |a_i]$ and $t''_i = [a_{i+1}| \cdots |a_n]$, for $i = 0, \dots, n$. 
Given $f \in \Hom_k(T,A)$, $a \in A$, and $t \in T$, we define
\begin{equation}
\label{left}
f \cdot (a \otimes t) = \sum_i  (-1)^{|a||t| + |t_i'|(|t''_i| + |a|)}  f(t_i'')a \otimes t_i'
\end{equation}
and 
\begin{equation}
\label{right}
(a \otimes t) \cdot f = \sum_i  (-1)^{|f| |t|}  a f(t_i') \otimes t_i''.
\end{equation}

Equivalently, left multiplication by $f$ is the function
\begin{equation}
\label{leftabstract}
\s \circ (\id_T \otimes \mu) \circ (\id_T \otimes f \otimes \id_A) \circ (\Delta \otimes \id_A) \circ \s,
\end{equation}
where  $\s: A \otimes_k T \to T \otimes_k A$ is the switching map given by
$\s(a \otimes t) = (-1)^{|a||t|} t \otimes a$. 
Recall that giving $A \otimes_k T$ the structure of a right module over $\Hom_k(T,A)$ is equivalent
to giving it the structure of a left module over $\Hom_k(T,A)^\op$, and that the two are related by
$$
(a \otimes t) \cdot f = (-1)^{|f|(|a| + |t|)} f^\op \cdot (a \otimes t),
$$
where, by $f^\op$, we mean the element $f \in \Hom_k(T, A)$ regarded as an element of $\Hom_k(T,A)^\op$. 
With this in mind, we can see that right multiplication by $f \in \Hom_k(T,A)$ on $A \otimes_k T$ 
corresponds to the left action of $f^\op$ given by the function
\begin{equation}
\label{rightabstract}
(\mu \otimes \id_T) \circ (\id_A \otimes f^\op  \otimes \id_T) \circ (\id_A \otimes \Delta).
\end{equation}

We also equip $A \otimes_k T$  with the differential $d_{\otimes} = d_A \otimes \id + \id \otimes \d_T$.

\begin{lem} 
\label{bimodule}
The actions described in 
\eqref{left} and \eqref{right} give $(A \otimes_k T, d_\otimes)$ the structure of a bimodule over the cdga $(\Hom_k(T,A), \star, d_{\Hom}, h \circ \e)$. 
By definition, this means
$A \otimes_k T$ is a $\Hom_k(T,A)-\Hom_k(T,A)$ bimodule in the usual sense, $d_\otimes$ satisfies the Leibniz rule for both the left and right actions, and
$$
d_{\otimes} \circ d_{\otimes} = [h \circ \e, -]
$$
where $[h \circ \e,-]$ is the endomorphism of $A \otimes_k T$ determined by the two actions given by
$$
[h \circ \e, \o] = (h \circ \e) \cdot \o -  \o \cdot (h \circ \e).
$$
\end{lem}

\begin{proof} 
The actions defined in \eqref{left} and \eqref{right} are obviously additive. It is a straightforward (but tedious) exercise to check the that they are associative, and also that the relation
$$
(f \cdot (a \otimes t)) \cdot g = f \cdot ((a \otimes t) \cdot g)
$$
holds. It is clear that $d_{\otimes} \circ d_{\otimes} = [h \circ \e, -]$, and the left and right Leibniz rules follow in a manner similar to the proof of Lemma~\ref{derivation}
outlined above; that is, one applies the abstract formulas \eqref{leftabstract} and \eqref{rightabstract} for the left and right actions, recalling that $d_A$ is a derivation and $\delta_T$ is a coderivation. 
\end{proof}

We now deform the differential on $(A \otimes_k T, d_\otimes)$ to make it a dg-module over the dga $\Hom^\tau_k(T,A)$.

\begin{lem} The endomorphism $d_\otimes - [\tau,-]$ of $A \otimes_kT$ makes it a left dg-module over the dga 
$\Hom^\tau_k(T,A)$.
\end{lem}

\begin{proof}
Proceed exactly as in the proof of Lemma~\ref{lem326}, applying Lemma~\ref{bimodule}.
\end{proof}

Let us write $A \otimes_k^\tau T$ for the complex $(A \otimes_k T, d_\otimes - [\tau,-])$.
Explicitly, the differential $d_\otimes - [\tau, -]$ sends an element $\a \in A \otimes_k T$ to
$$
 (d_A \otimes \id_T)(\a) +
(\id_A \otimes \d^{(2)}_T)(\a) +
(\id_A \otimes \d^{(1)}_T)(\a) +
(\id_A \otimes \d^{(0)}_T)(\a) - \tau \cdot \a + (-1)^{|\a|} \a \cdot \tau. 
$$
For $\a = a_0[a_1| \cdots|a_n]$ we have 
$$
\begin{aligned}
(d_A \otimes \id_T)(\a) & = d(a_0)[a_1 | \cdots | a_n], \\
(\id_A \otimes \d^{(2)}_T)(\a) & =  \sum_{j=1}^{n-1} (-1)^{|a_0| + |a_1| + \cdots |a_j| - j } a_0 [a_1 | \cdots |a_ja_{j+1} | \cdots | a_n], \\
(\id_A \otimes \d^{(1)}_T)(\a) & = 
\sum_{j=1}^n (-1)^{|a_0| + |a_1| + \cdots + |a_{j-1}| - j} a_0[a_1| \cdots | d(a_j) | \cdots | a_n], \\
(\id_A \otimes \d^{(0)}_T)(\a) & = 
\sum_{j=0}^n (-1)^{|a_0| + |a_1| + \cdots + |a_{j}| - j} a_0[a_1| \cdots | a_j|h | a_{j+1}|  \cdots | a_n]. \\
\end{aligned}
$$
In the formulas for the left and right actions of $\tau \in \Hom_k(T,A)$ on $A \otimes_k T$,  
the only terms that do not vanish are those when $i = n-1$ and $i = 1$, respectively, and so we have
$$
- \tau \cdot a_0[a_1| \cdots|a_n] = - (-1)^{(|a_0| + \cdots + |a_{n-1}| -n - 1)(|a_n|-1)}      a_n a_0 [a_1| \cdots |a_{n-1}],
$$
and
$$
\begin{aligned}
(-1)^{|\a|} a_0[a_1| \cdots|a_n] \cdot \tau 
& =  (-1)^{|a_0| + |a_1| + \cdots + |a_n| -n +|a_1| + \cdots + |a_n| -n} a_0a_1 [a_2| \cdots |a_{n}] \\
& =  (-1)^{|a_0|} a_0a_1 [a_2| \cdots |a_{n}].
\end{aligned}
$$

Putting these equations together gives that $d_\otimes - [\tau,-]$ 
coincides with the formula for the Hochschild differential $b = b_2 + b_1 + b_0$ for $A$ given in Section~\ref{MCunital}. That is,
we have an identity
$$
A \otimes_k^\tau T = (\tHoch(A), b) 
$$
of chain complexes.

In summary, the Hochschild homology of $A$ is $(\tHoch(A), b) = A \otimes_k^\tau T$, and it is a left dg-module over Hochschild cohomology of $A$,  which is
the dga $\Hom^\tau(T,A)$.

\section{Proof of Theorem \ref{existence}}
\label{appendixtechnical}
Assume $k$ is a field of characteristic $0$ and $\cC$ is a curved differential $\G$-graded category over $k$, where $\G = \Z$ or $\G = \Z/2$.
Define the graded $k$-vector space
$$
\cC/[\cC, \cC]= \bigoplus_X \End(X)/\sim,
$$
where $\sim$ is the equivalence relation given by modding out by commutators; that is, it is generated by
$\a \b \sim (-1)^{|\a||\b|} \b \a$ for all pairs of maps of the form $\a: X \to Y$ and  $\b: Y \to X$. 
The graded $k$-vector space $\cC/[\cC, \cC]$ is not an algebra, but it does enjoy the following ``(graded) cyclic invariance'' property: given composable morphisms
$$
X_m \xra{\a_m} X_{m-1} \xra{\a_m} \cdots  \xra{\a_1} X_{0} \xra{\a_0} X_{m} 
$$
we have
$$
\a_0 \circ \cdots \circ \a_m = (-1)^{|\a_m|(|\a_0| + \cdots + |\a_{m-1}|)} \a_m \circ \a_0 \circ \cdots \circ \a_{m-1} \pmod{[\cC, \cC]}.
$$

The differentials in $\cC$ induce
a map $\overline{d}$ on $\cC/[\cC, \cC]$ that squares to $0$, making $\cC/[\cC,\cC]$ into a complex of $k$-vector spaces. 
Let $v$ be a formal parameter of degree $-2$, and let  $\cC/[\cC, \cC] [[v]]$ be the complex of $k[[v]]$-modules
obtained from $\cC/[\cC,\cC]$ by extension of scalars along $k \to k[[v]]$. 
Define a $k[v]$-linear map of degree $0$ 
$$
\phi: \oHoch(\cC)[v] \to \cC/[\cC, \cC] [[v]]
$$
as follows:
for any integer $n \geq 0$, objects $X_0, \dots, X_n$ of $\cC$ and morphisms $a_i: X_{i+1} \to X_i$, where $X_{n+1} := X_0$, we set 
$$
\phi(a_0 [a_1 | \cdots |a_n]) = \sum_{\vj} \frac{a_0 h_1^{j_0} d(a_1) \cdots h_n^{j_{n-1}}d(a_n) h_0^{j_n}}{(J+n)!} v^{J+n}  
\pmod{[\cC, \cC]},
$$
where $h_0, \dots, h_n$ are the curvatures of $X_0, \dots, X_n$, $\vj = (j_0, \dots, j_n)$ ranges over all $(n+1)$-tuples of non-negative integers, and $J = j_0 + \cdots + j_n$. To ease notation in this section, we omit the bars over the $a_i$ when we write elements of the reduced Hochschild homology complex.

This map extends in the evident way to a map
$$
\oHoch^{II}(\cC)[[v]] \to \cC/[\cC, \cC] [[v]]
$$
which we also call $\phi$. In more detail, the first map is continuous for the topology on the source determined by the (reduced version of the) filtration
$F^j$
 defined in  \eqref{E721} and the usual topology for power series on the target, and hence it induces a map on completions.

\begin{prop} 
With the notation above, we have
$$ 
v \overline{d} \circ \phi =   - v \phi \circ b_2 +  \phi \circ B.
$$
\end{prop}

\begin{rem} The Proposition implies that $\phi \circ B$ is divisible by $v$; this is clear from the formula for $\phi$, since the only way to get a
  term not involving $v$ occurs when $n = 0$. 
\end{rem}

\begin{proof}[Proof of Proposition] 
We have
$$
\overline{d} (\phi(a_0[a_1| \cdots | a_n]))=
\beta(a_0[a_1| \cdots | a_n]) + \gamma(a_0[a_1| \cdots | a_n]),
$$
where
$$
\beta(a_0[a_1| \cdots | a_n]) = \sum_{j_0, \dots, j_n} \frac{d(a_0) h_1^{j_0} d(a_1) \cdots d(a_n) h_0^{j_n}}{(J+n)!} v^{J+n}, 
$$
and
$$
\begin{aligned}
\gamma(a_0[a_1| \cdots | a_n])  
& =  \sum_{j_0, \dots, j_n} \sum_{i=1}^n \s_i \frac{a_0 h_1^{j_0} d(a_1) \cdots d^2(a_i) \cdots d(a_n) h_0^{j_n}}{(J+n)!} v^{J+n} . \\
\end{aligned}
$$
Here, $\s_i \in \{\pm 1\}$ is given by the formula
$$
\s_i = \s_i(a_0, \dots, a_n) = (-1)^{|a_0| + |sa_1| + \cdots |sa_{i-1}|} =
(-1)^{|a_0| + |a_1| + \cdots |a_{i-1}| + i - 1}.
$$ 
Note that 
$\s_1 = (-1)^{|a_0|}$ and
$\s_{i+1} = - (-1)^{|a_i|} \cdot \s_i$ for $i = 1, \dots n-1$.

Our task is to prove
$$
v \g + v \b = - v \phi \circ b_2  + \phi \circ B.
$$
In fact, we show that
$$
\g = - \phi \circ b_2
\and
v \b = \phi \circ B.
$$

We first show $\g = - \phi \circ b_2$. Since $d^2(a_i) = h_{i}a_i - a_ih_{i + 1}$ we get
$$
\gamma(a_0[a_1| \cdots | a_n])  = \sum_{i=1}^n \s_i \gamma_i(a_0[a_1| \cdots | a_n]), 
$$ 
where
$$
\begin{aligned}
\gamma_i(a_0[a_1| \cdots | a_n]) = & \sum_{\vj} \frac{a_0 h_1^{j_0} d(a_1) \cdots h_{i}^{j_{i-1}+1} a_i h_{i + 1}^{j_{i}} \cdots d(a_n) h_0^{j_n}}{(J+n)!}  v^{J+n} \\
& - \sum_{\vj} \frac{a_0 h_1^{j_0} d(a_1) \cdots h_i^{j_{i-1}} a_i h_{i + 1}^{j_{i}+1} \cdots d(a_n) h_0^{j_n}}{(J+n)!} v^{J+n}\\
& = \sum_{\vj, j_{i} = 0} \frac{a_0 h_1^{j_0} d(a_1) \cdots h_i^{j_{i-1}+1} a_i d(a_{i+1}) \cdots d(a_n) h_0^{j_n}}{(J+n)!} v^{J+n}  \\
&  - \sum_{\vj, j_{i-1} = 0}  \frac{a_0 h_1^{j_0} d(a_1) \cdots d(a_{i-1}) a_i h_{i+1}^{j_{i}+1} \cdots d(a_n) h_0^{j_n}}{(J+n)!} v^{J+n}.\\
\end{aligned}
$$
Upon reindexing the first summation by $n$-tuples $\vl = (l_0, \dots, l_{n-1})$ with 
$$
l_0 = j_0, \dots, l_{i-1} = j_{i-1}+1, l_i = j_{i+1}, \dots, l_{n-1} = j_n,
$$
and similarly for the second summation, we obtain
$$
\begin{aligned}
\gamma_i(a_0[a_1| \cdots | a_n]) 
& = \sum_{\vl, l_{i-1} >0} \frac{a_0 h_1^{l_0} d(a_1) \cdots h_{i}^{l_{i-1}} a_i d(a_{i+1}) h_{i + 2}^{l_i} \cdots d(a_n) h_0^{l_{n-1}}}{(L-1+n)!} v^{L-1+n}  \\
&  - \sum_{\vl,  l_{i-1} > 0}  \frac{a_0 h_1^{l_0} d(a_1) \cdots d(a_{i-1}) a_i h_{i}^{l_{i-1}} \cdots d(a_n) h_{0}^{l_{n-1}}}{(L-1+n)!} v^{L-1+n},\\
\end{aligned}
$$
where $L = l_0 + \cdots + l_{n-1} = J + 1$. The terms added by omitting the restriction $l_{i-1} > 0$ in each summation cancel, and thus 
$$
\gamma_i(a_0[a_1| \cdots a_n]) = \gamma'_i(a_0 [a_1 | \cdots | a_n]) - \gamma''_i(a_0 [a_1 | \cdots | a_n]),
$$
where we define
\begin{equation} \label{E322}
\begin{aligned}
\gamma'_i(a_0 [a_1 | \cdots | a_n]) 
& := \sum_{\vl} \frac{a_0 h_1^{l_0} d(a_1) \cdots d(a_{i-1}) h_i^{l_{i-1}} a_i d(a_{i+1}) h_{i + 2}^{l_i} \cdots d(a_n) h_0^{l_{n-1}}}{(L-1+n)!} v^{L-1+n}  \\
 \gamma''_i(a_0 [a_1 | \cdots | a_n]) 
& := \sum_{\vl}  \frac{a_0 h_1^{l_0} d(a_1) \cdots d(a_{i-1}) a_i h_i^{l_{i-1}} d(a_{i+1})\cdots d(a_n) h_0^{l_{n-1}}}{(L-1+n)!} v^{L-1+n}.\\
\end{aligned}
\end{equation}
Note that
$$
\gamma'_n = \sum_{l_0, \dots, l_{n-1}} \frac{a_0 h_1^{l_0} d(a_1) \cdots d(a_{n-1})h_n^{l_{n-1}} a_n}{(L-1+n)!} v^{L-1+n} 
$$
and
$$
\gamma''_1 =  \sum_{l_0, \dots, l_{n-1}}  \frac{a_0 a_1 h_1^{l_0} d(a_2) \cdots d(a_{n})  h_0^{l_{n-1}} }{(L-1+n)!} v^{L-1+n}.
$$
Thanks to the cyclic invariance of $\cC/[\cC,\cC]$, the former can be rewritten as
$$
\gamma'_n = \sum_{l_0, \dots, l_{n-1}} (-1)^{|a_n|(|a_0| + |a_1|  + \cdots + |a_{n-1}| + n - 1)}
\frac{a_na_0 h_1^{l_0} d(a_1) \cdots d(a_{n-1})h_n^{l_{n-1}} }{(L-1+n)!} v^{L-1+n}. 
$$
In summary, we have
$$
\g = \sum_{i=1}^n \s_i \g_i = \sum_i \s_i (\g'_i - \g''_i),
$$
with $\g'_i$, $\g''_i$ defined in \eqref{E322}.

On the other hand, rewriting \eqref{b2formula} using  $\s_i$ gives 
\begin{equation} \label{E322b}
\begin{aligned}
\phi(b_2(a_0 [a_1| \dots| a_n])) 
&  =    \s_1 \phi( a_0a_1[a_2 | \cdots | a_n]) \\
& + \sum_{j=1}^{n-1}  \s_{j+1} \phi(a_0[a_1| \cdots | a_j a_{j+1}| \cdots | a_n]) \\
& - (-1)^{(|a_n| - 1)(|a_0| + \cdots + |a_{n-1}| -(n-1))} \phi(a_na_0[a_1| \cdots |a_{n-1}]). \\
\end{aligned}
\end{equation}
For $1 \leq i \leq n-1$, using $d(a_ia_{i+1}) = d(a_i) a_{i+1} + (-1)^{|a_i|} a_i d(a_{i+1})$ gives 
$$
\begin{aligned}
& \phi(a_0[a_1| \cdots | a_ia_{i+1} | \cdots | a_n]) \\
& =  \sum_{\vl} \frac{a_0 h_1^{l_0} d(a_1) \cdots h_i^{l_{i-1}} d(a_i a_{i+1}) h_{i + 2}^{l_i} \cdots d(a_n) h_0^{l_{n-1}}}{(L-1+n)!} v^{L-1+n}  \\ 
& =  \sum_{\vl} \frac{a_0 h_1^{l_0} d(a_1) \cdots h_i^{l_{i-1}} d(a_i) a_{i+1} h_{i + 2}^{l_i} \cdots d(a_n) h_0^{l_{n-1}}}{(L-1+n)!} v^{L-1+n}  \\
 &   +  (-1)^{|a_i|}  \sum_{\vl} \frac{a_0 h_1^{l_0} d(a_1) \cdots h_i^{l_{i-1}} a_i d(a_{i+1}) h_{i + 2}^{l_i} \cdots d(a_n) h_0^{l_{n-1}}}{(L-1+n)!} v^{L-1+n}  \\
& =  \g''_{i+1}  + (-1)^{|a_i|} \g_i' \\
\end{aligned}
$$
We also have
$$
\s_1 \phi( a_0a_1[a_2 | \cdots | a_n]) =
\s_1 \sum_{\vl}  \frac{a_0a_1 h_1^{l_0} d(a_2) \cdots d(a_n) h_0^{l_{n-1}}}{(L +n -1)!} v^{L+n-1} 
= \s_1 \gamma''_1,
$$
and
$$
\begin{aligned}
- & (-1)^{(|a_n| - 1)(|a_0| + \cdots + |a_{n-1}| -(n-1))}  \phi(a_na_0[a_1| \cdots |a_{n-1}]) \\
& = - (-1)^{|a_n| (|a_0| + \cdots + |a_{n-1}| -(n-1))} \s_n \sum_{\vl} \frac{a_na_0h_1^{j_0} d(a_1)  \cdots d(a_{n-1}) h_n^{j_{n-1}}}{(L+n-1)!} v^{L+n-1}  \\
& =- \s_n\gamma_n' .\\
\end{aligned}
$$
Substituting these equations into \eqref{E322b} and using 
$\s_{i+1} = - (-1)^{|a_i|} \cdot \s_i$ yields
$$
\begin{aligned}
\phi(b_2(a_0 [a_1| \dots| a_n])) 
& =  \s_1 \gamma''_1  + \sum_{i=1}^{n-1}  (\s_{i+1} \g''_{i+1}  - \s_{i} \g_i') - \s_n \gamma'_n \\
& = - \sum_{i=1}^n \s_i (\g'_i -  \g''_i) \\
& = - \gamma(a_0[a_1| \cdots |a_n]). \\
\end{aligned}
$$
This proves $\gamma = - \phi \circ b_2$.

We now prove  $v \b = \phi \circ B$. We have
$$
\begin{aligned}
& \phi(B(a_0[a_1| \cdots |a_n]) \\
& = 
\sum_{i=0}^n \tau_i \sum_{\vj} \frac{h_{i}^{j_0} d(a_i) h_{i + 1}^{j_1} d(a_{i+1}) \cdots h_{n}^{j_{n-i}} d(a_n) 
h_0^{j_{n-i+1}} d(a_0) h_1^{j_{n-i+2}} 
\cdots h_{i-1}^{j_n} d(a_{i-1}) h_i^{j_{n+1}})}{(J+n+1)!} v^{J+n+1} , \\
\end{aligned}
$$
where the signs $\tau_i$ are given by 
$$
\tau_i = 
(-1)^{(|sa_i| + \cdots +|sa_n|) (|sa_0| + \cdots +|sa_{i-1}|)   }
= (-1)^{(|a_i| + \cdots +|a_n| - n+i-1) (|a_0| + \cdots +|a_{i-1}| - i)   }.
$$
The cyclic invariance of $\cC/[\cC,\cC]$ gives
$$
\begin{aligned}
\phi(B(a_0[a_1| \cdots |a_n]) 
& =  \sum_{i=0}^n  \tau_i 
\sum_{\vm} \frac{h_i^{m_0} d(a_i) h_{i + 1}^{m_1} \cdots d(a_n) h_0^{m_{n-i+1}} d(a_0)h_1^{m_{n-i+2}} \cdots d(a_{i-1})h_{i}^{m_{n+1}}}{(M+n+1)!} v^{M+n+1} \\
& =  \sum_{i=0}^n  \sum_{\vm} \frac{d(a_0) h_1^{m_{n-i+2}} d(a_{1}) \cdots   d(a_{i-1})    h_{i }^{m_{n+1} + m_0} d(a_i) \cdots             
h_n^{m_{n-i}} d(a_{n}) h^{m_{n-i+1}}      } {(M+n+1)!} v^{M+n+1}, \\
\end{aligned}
$$
where $\vm = (m_0, \dots, m_{n+1})$ ranges over all $(n+2)$-tuples of non-negative integers, and $M := m_0 + \cdots + m_{n+1}$. 
In the inner summation indexed by $i$, make the substitutions 
$$
j_0 = m_{n-i+2} \text{, } j_1 = m_{n-i+3}  \text{, } \dots  \text{, } j_{i-2} = m_{n}  \text{, }
j_{i-1} = m_{n+1} + m_0  \text{, }
j_i = m_1  \text{, } \dots  \text{, } j_n = m_{n-i+1}
$$
to get
$$
\begin{aligned}
\phi(B(a_0[a_1| \cdots |a_n]) 
& =  \sum_{i=0}^n  \sum_{\vj} (j_{i-1} + 1)\frac{ d(a_0) h_1^{j_0} d(a_{1}) \cdots  h_i^{j_{i-1}} d(a_i) \cdots d(a_n) h_0^{j_{n}} }{(J+n+1)!} v^{J+n+1}
\\
& =  \sum_{\vj} (j_{i-1} + 1)  \sum_{i=0}^n  \frac{ d(a_0) h_1^{j_0} d(a_{1}) \cdots  h_i^{j_{i-1}} d(a_i) \cdots d(a_n) h_0^{j_{n}} }{(J+n+1)!} v^{J+n+1}
\\
& = \sum_{\vj}  (J + n + 1)   \frac{ d(a_0) h_1^{j_0} d(a_{1}) \cdots  h_i^{j_{i-1}} d(a_i) \cdots d(a_n) h_0^{j_{n}} }{(J+n+1)!} v^{J+n+1} \\
& =  \sum_{\vj}  \frac{ d(a_0) h_1^{j_0} d(a_{1}) \cdots  h_i^{j_{i-1}} d(a_i) \cdots d(a_n) h_0^{j_{n}} }{(J+n)!} v^{J+n+1} \\
& = v \b (a_0[a_1| \cdots |a_n])   \qedhere \\
\end{aligned} 
$$
\end{proof}

\begin{cor} \label{cor321}
With the notation of Theorem \ref{existence}, we have
$$
\tr_\n \circ (b_2 + uB) = u d \circ \tr_\n.
$$
\end{cor}

\begin{proof} Recall that $\cA = (A,0)$ is an essentially smooth curved $k$-algebra, and $\cD$ is a full cdg subcategory of $\qPerf(\cA)$ consisting of objects with trivial differentials.
From $\cD$ we form a new cdg category, $\cC$, that has the same objects as $\cD$ but with morphism spaces given by
$$
\Hom_{\cC}(P, P') := \Hom_A(P, P') \otimes_A \Omega^\bu_{A/k}[[u]],
$$
where $u$ has degree $2$. We equip these with the $k[[u]]$-linear differential $\delta$ given by $u [\n, -]$; that is,
$$
\delta(\a) =  u \n_{P'} \circ \a - (-1)^{|\a|} u \a \circ \n_P.
$$
For each $P$, set the curvature of $P$ in $\cC$ to be $u^2 \n_P^2$. These data make $\cC$ into a $k[[u]]$-linear curved differential $\G$-graded category.

Let $\cD^\nat$ and $\cC^\nat$ denote the $k$-linear categories underlying $\cD$ and $\cC$, obtained by forgetting differentials and curvatures.
There is an evident functor
$$
F: \cD^\nat \to \cC^\nat
$$
which induces a map
$$
F_*: \oHoch^{II}(\cD^\nat) \to \oHoch^{II}(\cC^\nat)
$$
that commutes with $b_2$ and $B$. Moreover, the collection of trace maps
$$
\End_A(P) \otimes_A \Omega^\bu_{A/k}[[u]] \to  \Omega^\bu_{A/k}[[u]] 
$$
induce a map 
$$
\pi: \cC/[\cC, \cC] [[v]] \to \Omega^\bu_{A/k}[[u, v]].
$$
Recall that, by Lemma \ref{lem326b},
$$
\pi \circ \overline{\delta} = u d \circ \pi,
$$ 
where $d$ is the de Rham differential. 

Thus, we get an induced map
$$
\pi \circ \phi \circ F_*: \oHoch^{II}(\cD^\nat) \to \Omega^{\bu}_{A/k}[[u, v]]
$$
given by
$$
\begin{aligned}
a_0[a_1|\cdots |a_n] & \mapsto   \sum_{\vj}  
\frac{\pi(a_0 (u^2\n^2)^{j_0}   u a'_1 (u^2 \n^2)^{j_1} \cdots  u a'_n (u^2 \n^2)^{j_n})}{(J+n)!}  v^{J+n} \\
& = \sum_{\vj}  \frac{\pi(a_0 \n^{2j_0}   a'_1 \n^{2j_1} \cdots   a'_n \n^{2j_n})}{(J+n)!}  u^{2J+n} v^{J+n}, \\
\end{aligned}
$$
where the $a_i'$ are as defined in Definition \ref{tracedefinition}. Notice that the coefficient of $u^Nv^M$ in $(\pi \circ \phi \circ F_* )(a_0[a_1 | \cdots | a_n])$ is  $0$ when $N$ or $M$ is greater than $\dim(A)$, and so $\pi \circ \phi \circ F_*$ takes values in $\Omega^{\bu}_{A/k}[u,v]$. Composing $\pi \circ \phi \circ F_*$ with the composition
$$
l: \Omega^{\bu}_{A/k}[u, v] \to \Omega_{A/k}^{\bu}[u, u^{-1}] \into \Omega^{\bu}_{A/k}[[u, u^{-1}]],
$$
where the first map sends $v$ to $-u^{-1}$ and the second is the inclusion, we obtain a map 
$$
\widetilde{\tr}_\n: \oHoch^{II}(\cD) \to \Omega^{\bu}_{A/k}[[u, u^{-1}]]
$$
given by the formula
$$
\widetilde{\tr}_\n(a_0[a_1| \cdots|a_n]) = \sum_{\vj}  (-1)^{J+n} \frac{\pi(a_0 \n^{2j_0}   a'_1 \n^{2j_1} \cdots   a'_n \n^{2j_n} )}{(J+n)!}  u^{J}. 
$$
In other words, $\widetilde{\tr}_\n$ differs from the map $\tr_\n$ defined in \eqref{trdef} by a sign of $(-1)^n$. By the Proposition and the calculations above, we have
\begin{align*}
u d \widetilde{\tr}_\n  &= u d l \pi \phi F_* \\
&= l u d  \pi \phi F_*  \\
&=l \pi \overline{\delta} \phi F_* \\
&= ul \pi v \overline{\delta} \phi F_* \\
 &= ul\pi(-v \phi b_2 + \phi B)F_* \\
 &= ul\pi(-v \phi F_* b_2 + \phi F_* B) \\
 &= - \widetilde{\tr}_\n (b_2 + u  B).
\end{align*}
Since $b_2$ and $u B$ each increase the index $n$ by one, it follows that
$$
\tr_\n \circ (b_2 + uB) = u d \circ \tr_\n.   \qedhere
$$
\end{proof}

For a cdg category $\cC$, a {\em central element of degree $2$} is a collection $z$ of degree $2$ elements
$$
z_X \in \End_{\cC}(X)
$$
ranging over all $X \in \cC$ such that, 
for all morphisms $\a: X \to Y$, we have $\a z_X = z_Y \a$  
and $\a d(z_X) = (-1)^{|\a|} d(z_Y) \a$. 
Given such a collection $z$, we define a degree $1$ endomorphism
$b_0^z$ of $\Hoch^{II}(\cC)$ by the formula
$$
b^z_0(a_0 [a_1| \dots| a_n]) 
 =   \sum_{j=0}^n (-1)^{|a_0| +|a_1| + \cdots + |a_{j}| - j } a_0 [a_1| \cdots |a_{j}| z_{j+1} |a_{j+1}| \cdots | a_n],
$$
where if $\a_{j+1}: X_{j+2} \to X_{j+1}$ and $\a_{j}: X_{j+1} \to X_j$, then $z_{j+1} = z_{X_{j+1}}$.

Write $d(z)$ for the collection of degree $3$ elements $d(z_X) \in \End_\cC(X)$, $X \in \cC$. 
Left multiplication by $d(z)$ determines a degree $3$ endomorphism of
$\bigoplus_X \End_\cC(X)$. Given maps $\a: X \to Y$ and $\b: Y \to X$, we have 
$$
\begin{aligned}
d(z_X) \cdot [\alpha, \beta] 
& = d(z_Y) \a \b - (-1)^{|\a||\b|} d(z_X) \b \a  \\
& = d(z_Y) \a \b - (-1)^{|\a||\b| + |\b|} \b d(z_Y)  \a   \\
& = d(z_Y) \a \b - (-1)^{(|d(z_Y)| + |\a|)|\b|} \b d(z_Y)  \a   \\
& = [d(z_Y) \a, \b]. \\
\end{aligned}
$$
In particular, left multiplication by $d(z)$ descends to a degree 3 endomorphism of $\cC/[\cC,\cC]$ which we shall write as $\l_{d(z)}$.

\begin{prop} For any central curvature element $z$ of $\cC$, we have
$$
\phi \circ b_0^z = v \l_{d(z)} \circ \phi
$$
\end{prop}

\begin{proof} For $a_i: X_{i+1} \to X_i$, $i = 0, \dots n$, with $X_{n+1} = X_0$, set $z_i = z_{X_i}$ and $h_i = h_{X_i}$.  Then
$$
\begin{aligned}
\phi & (b^z_0(a_0[a_1| \cdots| a_n])) \\
& = \phi(\sum_{i=0}^n (-1)^{|a_0| + \cdots + |a_i| - i} a_0[a_1| \cdots |a_i| z _{i+1} |a_{i+1}| \cdots | a_n])) \\
& = \sum_i (-1)^{|a_0| + \cdots + |a_i| - i}   \sum_{\vm} \frac{a_0 h_1^{m_0} d(a_1) \cdots d(a_i) h_{i+1}^{m_i}   d(z_{i+1})  h_{i+1}^{m_{i+1}} 
d(a_{i+1})  \cdots d(a_n) h_0^{m_{n+1}}}{(M+n+1)!}  v^{M +n + 1} \\
& = d(z_0) \sum_i \sum_\vm  \frac{a_0 h_1^{m_0} \cdots d(a_i) h_{i+1}^{m_i+m_{i+1}} d(a_{i+1}) \cdots d(a_n) h_0^{m_{n+1}}}{(M+n+1)!} v^{M+n+1} \\
\end{aligned}
$$
where $\vm = (m_0, \dots, m_{n+1})$ ranges over all $(n+1)$-tuples, and $M := m_0 + \cdots + m_{n+1}$. 
By substituting
$$
j_0 = m_0  \text{, } \dots  \text{, } j_{i-1} = m_{i-1}  \text{, } j_i = m_{i} + m_{i+1}  \text{, } j_{i+1} = m_{i+2} \text{, } \dots  \text{, } j_n = m_{n+1},
$$
we obtain
$$
\begin{aligned}
\phi & (b^z_0(a_0[a_1| \cdots| a_n])) \\
& = d(z_0)  \sum_{\vj} \sum_i (j_i+1) \frac{a_0 h^{j_0} \cdots d(a_i) h^{j_i} d(a_{i+1}) \cdots   d(a_n) h^{j_n} }{(J+n+1)!} v^{J+n+1} \\
& = d(z_0)  \sum_\vj (J+n+1)\frac{a_0 h^{j_0} \cdots d(a_i) h^{j_i} d(a_{i+1}) \cdots  d(a_n) h^{j_n} }{(J+n+1)!} v^{J+n+1} \\
& = v d(z_0)  \sum_\vj \frac{a_0 h^{j_0} \cdots d(a_i) h^{j_i} d(a_{i+1}) \cdots  d(a_n) h^{j_n} }{(J+n)!} v^{J+n} \\
& = v \l_{d(z)}  \phi(a_0[a_1| \cdots| a_n]) .   \qedhere \\
\end{aligned}
$$
\end{proof}

\begin{cor} \label{cor321b}
With the notation of Theorem \ref{existence}, we have
$$
\tr_\n \circ b_0 = \l_{dh} \circ \tr_\n.
$$
\end{cor}

\begin{proof} We adopt the notation in the proof of  Corollary \ref{cor321}. The curvature element $h \in A^2$ determines a central element of $\cC$ by setting $z_P
  \in \End_A(P) \subseteq \End_A(P) \otimes_A \Omega^\bu_{A/k}$ to be multiplication by $h$ on $P$. Note that $d(z_P)$ is left multiplication by $dh
  \in \Omega^1_{A/k}$. 
Using the Proposition, we have an equality of maps
$$
\pi \circ \phi \circ b^z_0 \circ F_* = v \pi \circ \l_{dh} \circ \phi \circ F_*
$$
from $\oHoch^{II}(\cD)$ to $\Omega^\bu_{A/k}[[u,v]]$. We have $b^z_0 \circ F_* = F_* \circ b_0$ and $\pi \circ \l_{dh} = \l_{dh} \circ \pi$, and so, 
upon
setting $v = -u^{-1}$,  we arrive at
$$
\widetilde{\tr}_\n \circ b_0 = - u^{-1} \l_{dh} \widetilde{\tr}_\n.
$$
As before, since $b_0$ shifts the index $n$ by $1$, this yields
$$
u \tr_\n \circ b_0 = \l_{dh} \circ \tr_\n. \qedhere
$$
\end{proof}

Observe that Corollaries \ref{cor321} and \ref{cor321b} together establish Theorem \ref{existence}.

\bibliographystyle{amsalpha}
\bibliography{Bibliography}

\end{document}